\documentclass{amsart}

\usepackage{lmodern}
\usepackage{textcomp}
\usepackage{latexsym,amssymb}
\usepackage{amsbsy}
\usepackage{amstext}
\usepackage{esint}
\usepackage{enumerate}
\usepackage{cancel,cite}
\usepackage{bbm}
\usepackage{mathrsfs}
\usepackage{ulem,lipsum}

\usepackage{hyperref}
\hypersetup{
  colorlinks=true,
  linkcolor=blue,
  urlcolor=black,
  citecolor=blue
}
\usepackage{xcolor}

\usepackage[margin=1.46in]{geometry}

\allowdisplaybreaks
\newtheorem{thm}{Theorem}[section]
\newtheorem{prop}[thm]{Proposition}
\newtheorem{corol}[thm]{Corollary}
\newtheorem{lemma}[thm]{Lemma}
\newtheorem{defn}[thm]{Definition}
\newtheorem{rem}[thm]{Remark}
\numberwithin{equation}{section}
\allowdisplaybreaks[4]

\DeclareMathAlphabet{\mathdutchcal}{U}{dutchcal}{m}{n}

\newlength{\defbaselineskip}
\def\vs{\vspace{1mm}}

\def \R {{\mathbb {R}}}
\def \eps{{\varepsilon}}
\def \p{{\phi}}
\def \Lc{{\mathcal{L}}}
\def \r{{\mathbb{R}}}
\def \Hd{{\mathdutchcal{H}}}
\def \Ec{{\mathcal{E}}}
\def \Ud{{\mathdutchcal{U}}}
\def \reduit{{\mathcal{R}}}
\def \balayage{{\widehat{\mathcal{R}}}}

\renewcommand{\Gamma}{\varGamma}

\begin{document}


\title[Totally degenerate differential operators]{The Dirichlet problem  for a family of \\ totally degenerate differential operators}


\author[M. Manfredini]{Maria Manfredini}  \address[M. Manfredini]{Dipartimento di Scienze Fisiche, Informatiche e Matematiche,
\newline\indent Universit\`a degli Studi di Modena e Reggio Emilia
 \newline\indent Via G. Campi 213/B, 41121 Modena, Italy} \email{\url{maria.manfredini@unimore.it}}

\author[M. Piccinini]{Mirco Piccinini}  \address[M. Piccinini]{Dipartimento di Matematica, \newline\indent Universit\`a di Pisa,
\newline\indent ~L.go~B.~Pontecorvo~5, 56127, Pisa, Italy}
\email{\url{mirco.piccinini@dm.unipi.it}}

\author[S. Polidoro]{Sergio Polidoro}  \address[S. Polidoro]{Dipartimento di Scienze Fisiche, Informatiche e Matematiche, \newline\indent Universit\`a degli Studi di Modena e Reggio Emilia,
  \newline\indent Via G. Campi 213/B, 41121 Modena, Italy} \email{\url{sergio.polidoro@unimore.it}}


\makeatletter
\@namedef{subjclassname@2020}{\textup{2020} Mathematics Subject Classification}
\makeatother


 \begin{abstract}
   In the framework of  Potential Theory we prove existence
   for the  Perron-Weiner-Brelot-Bauer solution
   to the Dirichlet problem related to a family of totally degenerate, in the sense of Bony, differential operators. We also state and prove a Wiener-type criterium and an exterior cone condition for the regularity of a boundary point.
  Our results apply to a wide family of strongly degenerate operators that includes the following example~$\Lc = t^2\Delta_x + \langle x, \nabla_y \rangle -\partial_t$, for~$(x,y,t) \in \R^N \times \R^{N} \times \R$.
\end{abstract}

\maketitle

\section{Introduction}

We consider second-order partial differential operators of the form
\begin{equation} \label{e-Kolm}
   \Lc 
   := t^{2\vartheta} \sum_{i=1}^m \partial^{2}_{x_i}+\sum_{i,j=1}^N b_{ij}x_j\partial_{x_i}-\partial_t,
\end{equation}
where~$z=(x,t) \in \R^{N+1}$,~$m, N$ and~$\vartheta$ are non-negative integers with $1 \le m \le N$. Moreover, $b_{ij}$ is a real constant for every $i,j=1, \dots, N$. The standing assumption of this article is:

\medskip
\noindent
{\bf[H.1]} \ {The matrix $B:=(b_{ij})_{i,j=1,\ldots,N}$ has the form
   \begin{equation*}
 			   B = 
 			    \begin{bmatrix}
 			       {0} &  {0} & \ldots & {0} & {0}  \\
 			       B_1 & {0} &  \ldots & {0} & {0} \\
                   {0} & B_{2}  & \ldots & {0} & {0} \\
                   \vdots & \vdots  & \ddots & \vdots & \vdots  \\
                   {0} & {0} & \ldots & B_{\kappa} & {0} 
   			 \end{bmatrix}
\end{equation*}
where every block $B_j$ is a $m_{j} \times m_{j-1}$ matrix of rank $m_j$ with $j = 1, 2, \ldots, \kappa$. Moreover, every $m_j$ is a positive integer such that
$$
m_0 \geq m_1 \geq \ldots \geq m_\kappa \geq 1, \quad \textrm{and} \quad m_0+m_1+\ldots+m_\kappa=N.
$$
We agree to let $m_0 := m$ to have a consistent notation, moreover every ${0}$ denotes a block matrix whose entries are zeros.}

\medskip 

As we will see in the following Proposition \ref{prop-hypo}, the condition {\bf[H.1]} implies that the operator $\Lc$ is \textit{hypoelliptic}. This means that, for every open set $U \subseteq \r^{N+1}$, every function $u \in L^1_{\text{loc}}(U)$, which solves the equation $\Lc u = f$ in the distributional sense, belongs to $C^\infty(U)$ whenever $f \in C^\infty(U)$. In particular, $u$ is a classical solution to $\Lc u = f$. 
\vspace{2mm}

In the framework of  Potential Theory, we address the boundary value problem
\begin{equation} \label{eq-bvp}
 	\begin{cases}
 	\Lc u   = 0 & \quad \text{in} \ U,\\
         u = \p & \quad \text{in} \  \partial U\,,
 \end{cases}
\end{equation}
where $U$ is any open subset of $\R^{N+1}$ and $\p \in {C}_c(\partial U)$\footnote{We indicate with ${C}_c(\partial U)$ the family of continuous functions on $\partial U$ with compact support}. The Perron-Weiner-Brelot-Bauer method provides us with a function $H^U_\p$ which is a classical solution to $\Lc H^U_\p = 0$ in $U$.

\begin{thm}
\label{main result}
Every open set $U \subseteq \R^{N+1}$ is resolutive, i.e. for every $\p \in {C}_c(\partial U)$ it is defined the {Perron-Weiner-Brelot-Bauer} solution $H^U_\p$ to the problem \eqref{eq-bvp}. Moreover $H^U_\p \in C^{\infty}(U)$ is a classical solution to $\Lc H^U_\p = 0$.
\end{thm}

Concerning the boundary value datum, it is well known that solution $H^U_\p$ to \eqref{eq-bvp} does not attain the prescribed boundary datum at every point of $\partial U$. We say that a point $z_{\rm o} \in \partial U$ is \textit{$\Lc$-regular} if $H^U_\p(z) \to \p(z_{\rm o})$ as $z \to z_{\rm o}$, for every $\p \in {C}_c(\partial U)$. 
The second main result of this article is a Wiener criterium for the regularity of a boundary point $z_{\rm o} \in \partial U$. Its statement requires some notation. For any fixed $\lambda \in (0,1)$ and for every $n \in \mathbb{N}$ we consider the following set
\begin{equation}
	\label{level set}
U^c_n(z_{\rm o}):= \Big\{z \in \r^{N+1} \setminus U :  \lambda^{-n\log n} 
\leq {\Gamma(z_{\rm o};z)} \leq \lambda^{-(n+1)\log(n+1)}\Big\} \cup \{z_{\rm o}\}.
\end{equation}
Here $\Gamma$ is the fundamental solution of $\Lc$, whose explicit expression is given in \eqref{fund-sol}. Moreover, {we denote with} $\balayage^1_{U^c_n(z_{\rm o})}$ 
the \textit{balayage} of the constant function $1$ on the set $U^c_n(z_{\rm o})$; see {forthcoming} Definition \ref{balayage}. With this bit notation we have
\begin{thm}
\label{main result2}
	Let $U \subset \r^{N+1}$ be an open set and let ${z_{\rm o}=(x_{\rm o},0)} \in \partial U$. Then $z_{\rm o}$ is an $\Lc$-regular point, if and only if
$$
 \sum_{n =1}^{\infty} \, \balayage^1_{U^c_n(z_{\rm o})}(z_{\rm o}) = + \, \infty.
$$
\end{thm}

The proof of Theorem \ref{main result2} is based on the explicit expression of the fundamental solution $\Gamma$ of $\Lc$ and follows the lines of the work \cite{KLT18} of {Kogoj, Lanconelli and Tralli}, where the regularity of Kolmogorov operator \eqref{e-Kolm} with $\vartheta=0$ is studied. In particular the article \cite{KLT18} extends to degenerate Kolmogorov equations the 
Wiener-Landis test
for the heat equation \cite{EG82}, and a regularity criterion proved by {Landis} in \cite{Lan69}, which again holds for the heat equation.

Finally we give a Zaremba-type criterion for the regularity of boundary points $z_{\rm o} =(x_{\rm o}, 0)$, which is a sufficient geometric condition relying on the definition of \textit{cone} $\mathcal{C}({z_{\rm o}})$ with vertex at $z_{\rm o}$; see Definition \ref{l-cone} {below}. 
\begin{prop}
	\label{main result3}
	Let $U \subset \r^{N+1}$ be an open set and let $z_{\rm o}=(x_{\rm o},0) \in \partial U$. If there exists an exterior cone $\mathcal{C}({z_{\rm o}})$ with vertex at $z_{\rm o}$, then $z_{\rm o}$ is $\Lc$-regular.
\end{prop}

Note that Proposition \ref{main result3} extends the analogous result proved by Manfredini \cite{Man97} for the case $\vartheta =0$. The requirement that the time coordinate of $z_{\rm o}$ is $t_{\rm o}= 0$ {both in Theorem \ref{main result2} and in Proposition \ref{main result3}} is needed because the definition of the cone $\mathcal{C}({z_{\rm o}})$ requires a dilation-invariance property of $\Lc$ which, in the case $\vartheta >0$, is granted only for $t_{\rm o}=0$ (see \eqref{dilation_invariance} below) {as well as to use the particular invariance properties of the fundamental solution $\Gamma$ of $\Lc$. On the other hand, the exterior cone criterion proved by Manfredini \cite[Theorem 6.3]{Man97} as well as the equivalent characterization of the regularity of boundary points in \cite[Theorem 5.4]{LU10}  do apply to every boundary point $z_{\rm o}=(x_{\rm o},t_{\rm o})$ with $t_{\rm o} \ne 0$.

\vspace{2mm}
 Let us briefly discuss Proposition~\ref{main result3} in the simplest case of~${\it m}=N$,~$\vartheta = 1$ and~$B = 0$, that is 
\begin{equation}\label{tot deg heat op}
	\Lc = t^2\Delta_x -\partial_t.
\end{equation}
	In this setting, for any~$(x,t) \in \r^{N+1}$ and any~$r>0$, the group of dilations is defined as
	$$
	\delta_r(x,t):= (r^3x,r^2t),
	$$
	and following Definition~\ref{l-cone} the cone of vertex $z_{\rm o} :=(x_{\rm o},0)$, height $T>0$ and base $K \subset \r^N$ is given by
	$$
	\mathcal{C}({z_{\rm o}}) :=\left\{(x_{\rm o} + r^3 x,-r^2T): x \in K, 0 \leq r \leq 1\right\}.
	$$
	Note that~\eqref{tot deg heat op} can be reduced to the heat equation by the change of the time-scale $u(x,t):= v(x,t^3/3)$. The classical parabolic cone 
	$$
	\widetilde{\mathcal{C}}({z_{\rm o}}) :=\left\{(x_{\rm o} + r x,-r^2\tilde{T}): x \in \tilde{K}, 0 \leq r \leq 1\right\},
	$$
	introduced in \cite{EK70} by Effros and Kazdan, guarantees the regularity of boundary point $z_{\rm o}$ for the solution $v$ to problem~\eqref{eq-bvp} relevant to the heat operator. Inverting the time-scale change of variables defined above $\widetilde{\mathcal{C}}({z_{\rm o}})$ does coincide with $\mathcal{C}({z_{\rm o}})$. However, this simple argument does not apply to ultraparabolic operators of the type~\eqref{e-Kolm}. Hence, the result stated in Proposition~\ref{main result3} can not be proved trivially with a time-scale change of variables in the general setting we are dealing with.
    \vs

We next give some comments about our main results. The first one concerns the uniqueness of the solution to the Dirichlet problem \eqref{eq-bvp}. A first simple answer to the uniqueness problem plainly follows from the maximum principle (see Corollary \ref{open set is mp-set} below). In particular, it implies that if $u$ and $v$ belong to $C(\overline U)$, for some bounded open set $U$, $u$ and $v$ are both classical solutions to \eqref{eq-bvp}, and attain the same values on $\partial U$, then necessarily agree. This result is however unsatisfactory. Indeed, it is well known that, if we consider the Cauchy-Dirichlet problem for the heat equation in a cylinder, then the solution is uniquely defined by the boundary condition on the \textit{parabolic boundary} of the cylinder. For this reason, we would expect that only the regular boundary points need to be considered in order to have the uniqueness of the solution to \eqref{eq-bvp}. Unfortunately, this fact is not true even in the case of the heat equation. Indeed, Luke\v{s} proves in Example 3.2 (D) of \cite{Luk74} that there exist bounded open sets that admit different solutions that agree at every regular boundary point.

The classical Perron method for the Laplace equation relies on the Poisson kernel, which provides us with the solution to the Dirichlet problem on any ball of the Euclidean space. In the more general setting of the abstract Potential Theory the Euclidean balls are replaced by the \textit{resolutive} sets, that are sets such that the {{Perron-Weiner-Brelot-Bauer}} solution is defined. Specifically, it is assumed that {\it there exists a family of resolutive open sets $\left\{ U_i \right\}_{i \in I}$, such that $\left\{ U_i \right\}_{i \in I}$ is a basis for the topology of the space}. Note that, unlike in the case of the Laplace equation, it is not required that all the boundary points of a resolutive set are regular.

The development of  Potential Theory is simpler in the case of the existence of a basis of {\it regular sets}, which are resolutive sets whose all the boundary points are regular. For this reason, even in the case of the heat operator, an effort has be done in order to build a basis of regular sets for the space $\R^{N+1}$. 
In particular, Bauer first pointed out in \cite{Bau66} that the cones defined for $(x_{\rm o},t_{\rm o}) \in \R^{N+1}$, and $r>0$ as
\begin{equation*}
 K_r(x_{\rm o},t_{\rm o}) = \big\{ (x,t) \in \R^{N+1} : |x-x_{\rm o}| < t_{\rm o}-t < r \big\}, 
\end{equation*}
have this property. Later Effros and Kazdan introduce in \cite{EK70} regular sets that are build as follows. Every set is the union of the cylinder 
 \begin{equation*}
 \widetilde Q_r(x_{\rm o},t_{\rm o}) = \big\{ (x,t) \in \R^{N+1} : |x-x_{\rm o}| < r, t_{\rm o}- 2 r < t \le t_{\rm o} - r \big\}, 
\end{equation*}
and the cone $K_r(x_{\rm o},t_{\rm o})$.
The regularity of the boundary points for the above families of sets is  proved by a simple barrier argument, which relies on the fact that, for every point of the lateral boundary of the cone, the spatial component $\nu_x$ of the outer normal $\nu = \left(\nu_x, \nu_t\right) \in \R^{N+1}$ is non zero.

Bony considers in \cite{Bon69} the boundary value problem \eqref{eq-bvp} for degenerate operators in the form 
\begin{equation}\label{e1}
     \Lc = \sum_{j=1}^m X_j^2  + Y,
\end{equation}
where $X_1, \dots, X_m$ and $Y$ are vector fields defined in {a domain $\Omega \subset \mathbb{R}^{N+1}$}, with smooth coefficients, satisfying H\"ormander's condition \cite{Horm67}
\begin{equation}\label{e-Hormander}
  {\rm Lie} (X_1, \ldots, X_m, Y)(z) = \R^{N+1}, \qquad \text{for every} \ z \in \Omega.
\end{equation}
We recall that ${\rm Lie} (X_1, \ldots, X_m, Y)$ is the Lie algebra generated by the vector fields $X_1, \ldots, X_m$ and $ Y$, that is the vector space generated by  $X_1, \ldots, X_m, Y$ and their commutators. The commutator of two given vector fields $W$ and $Z$ is the vector field defined as:
\begin{equation*}
	[W,Z] := W \, Z -  Z \, W.
\end{equation*}
In his work, Bony restricts his study to \textit{non totally degenerate} operators. This means that, for every $z \in \R^{N+1}$, at least one of the vector fields  $X_1(z), \ldots, X_m(z)$ is non zero. The \textit{non total degeneracy} of the operator $\Lc$ allows Bony to build a family of bounded open regular sets by a general method that relies on a barrier argument. Note that for $\vartheta =0$ the Bony's theory applies to the operator $\Lc$. We refer the reader to \cite{CL09}, \cite{Kog17}, \cite{LP94} and \cite{Man97} for the study of the relevant Dirichlet problem. 

We remark that the non total degeneracy of the operator $\Lc$ is a mild requirement. Indeed, from the very definition of commutator it follows that 
\begin{equation*}
	W(z) = 0 \ \text{and} \ Z(z) = 0 \quad \Rightarrow \quad [W,Z] (z) = 0,
\end{equation*}
thus 
\begin{equation*}
  X_1(z) = 0, \ldots, X_m(z) = 0 , Y(z) = 0 \quad \Rightarrow \quad {\rm Lie} (X_1, \ldots, X_m, Y)(z) = \{0 \}.
\end{equation*}
In particular, if $\Lc$ satisfies H\"ormander's condition, then at least one of the vector fields $X_1, \ldots, X_m, Y$ is different from zero. Concerning the operator $\Lc$, it can be written in the form \eqref{e1} with
\begin{equation*}
 X_j := t^\vartheta \partial_{x_j},\quad j=1,\ldots,m, \qquad Y := \langle B x, D \rangle -\partial_t,
\end{equation*}
and, as we say in Proposition \ref{prop-hypo}, the assumption {\bf[H.1]} is equivalent to H\"ormander's condition, even though $\Lc$ is \textit{totally degenerate} at $t=0$, for $\vartheta \geq 1$.

\vspace{2mm}
In this work we rely on the construction of the {Perron-Weiner-Brelot-Bauer} solution to the Dirichlet problem \eqref{eq-bvp} based on the existence of a family of \textit{resolutive sets}, as explained in the monograph \cite{CC72} by Constantinescu and Cornea. We recall that a family of resolutive sets for operators in the form \eqref{e-Kolm} satisfying the assumption {\bf[H.1]} has been built by Montanari in \cite{Mon96}. We point out that in the particular case of the heat operator, these resolutive sets agree with the standard cylinders.

\subsection*{Outline of the article} In Section \ref{sec:2} we specify the notation adopted throughout the rest of the article and recall some known results about the operator  $\Lc$. Moreover, we also give a detailed proof of the hypoellipticity of $\Lc$. 
In Section \ref{sec:3}  we recall all the notions and results from Potential Theory that we need.
We also give a  characterization of boundary regularity in the abstract setting of Potential Theory; see forthcoming Theorem \ref{sec2 lemma}. In Section \ref{sec:4}  we  construct the {{Perron-Weiner-Brelot-Bauer}} solution of the problem (\ref{eq-bvp})
and prove Theorem \ref {main result}. In Section \ref{sec:5}  we prove Theorem \ref{main result2} and  Proposition \ref{main result3}. 

\subsection*{Aknowledgements} We thank E. Lanconelli for his interest in our work and for bringing our attention to a meaningful example in reference \cite{Luk74}.

\subsection*{Funding} M.~Piccinini is supported by PRIN 2022 PNRR Project ``Magnetic skyrmions, skyrmionic bubbles and domain walls for spintronic applications'', PNRR Italia Domani, financed by EU via NextGenerationEU CUP\_D53D23018980001.

\section{Preliminaries}\label{sec:2}
In this section we specify the notation adopted throughout the rest of the paper and provide some known results about the family of operators we are dealing with. We also give a detailed proof of the hypoellipticity of $\Lc$ and of the existence of its fundamental solution.

{
We denote with~$c$ a positive universal constant greater than one, which may change from line to line. For the sake of readability, dependencies of the constants will be often omitted within the chains of estimates, therefore stated after the estimate. Relevant dependencies on parameters will be emphasized by using parentheses.

For any~$U \subset \mathbb{R}^{N+1}$ we denote with~$|U|$ the Lebesgue measure of $U$. 
As customary, for any~$r>0$ and any~$y_{\rm o} \in \mathbb{R}^{N+1}$ we denote by~$
B_r(y_{\rm o}) \equiv B(y_{\rm o};r):=\{y \in \mathbb{R}^{N+1} \,:\, |y-y_{\rm o}|< r\}$\,,
the open ball with radius~$r$ and center~$y_{\rm o}$.} Here and in the following of this note we write the operator $\Lc$ in H\"ormander's form
\begin{equation*} 
   \Lc = \sum_{j=1}^m ( t^{\vartheta} \partial_{x_j})^2  + \langle B x, \nabla \rangle -\partial_t = \sum_{j=1}^m X_j^2  + Y, 
\end{equation*}
with 
$$
	X_j := t^\vartheta \partial_{x_j}, \qquad  Y := \langle B x, \nabla \rangle -\partial_t,
$$
for $j=1, \dots, m$. As usual in the theory of H\"ormander's operators, we identify any vector field $X$ with the vector valued function whose entries are the coefficients of $X$, specifically
$$
	X = \sum_{j=1}^{N} c_j(x,t) \partial_{x_j} + c_{\rm o}(x,t) \partial_t \simeq (c_1(x,t), \dots, c_N(x,t), c_{\rm o}(x,t)).
$$
We denote by $A$ and $e^{-sB}$ the $N \times N$ matrices defined as
\begin{equation} \label{e^{-tB}}
	A :=
	\begin{bmatrix}
		I_{m_0} & {0}\\
		{0} & {0}
        \end{bmatrix}
	\, \qquad e^{-sB} := \sum_{n=0}^\infty \frac{(-1)^ns^n}{n!}B^n\,,
\end{equation}
where $I_{m_0}$ is the $m_0 \times m_0$ identity matrix and $s$ is any real number. For every $t, \tau \in \R$ we eventually define the matrix
\begin{equation} \label{c}
	C(\tau,t):= \int_{0}^{t-\tau} (t-s)^{2\vartheta}e^{-sB}A e^{-sB^T}{\rm d}s.
\end{equation}

We spend few words about some geometric aspects related to the operator $\Lc$. In the article \cite{LP94} the composition law 
\begin{equation} \label{eq-liegroup}
	(x,t) \circ (\xi,\tau) = (\xi+e^{-\tau B}x,t+\tau) \quad  (x,t),(\xi,\tau) \in \r^{N+1},
\end{equation}
was introduced and it was  proved that $\mathbb{K}:= (\r^{N+1},\circ)$ is a non-commutative Lie group with zero element $(0,0)$ and inverse element given by 
$$ 
 (x,t)^{-1}=(-e^{tB}x,-t) \quad \forall (x,t) \in \r^{N+1}.
$$
Moreover, the operator $\Lc$ with $\vartheta = 0$ is invariant with respect to the left translation \eqref{eq-liegroup}. We refer the reader to the monograph \cite{BLU07} for a general presentation of the theory of Lie groups and to \cite{AP20,APR23} for a survey of results on the operator $\Lc$. However, the operator $\Lc$ is not translation invariant as $\vartheta \ge 1$; see Proposition 1.2.13 in \cite{BLU07}. Nevertheless, the matrix \eqref{e^{-tB}} will be used also for $\vartheta \ge 1$ in order to define a basis of resolutive sets and to state a Harnack inequality.

For every $r>0$, we denote with $\delta_r: \r^{N+1} \to \r^{N+1}$ the {%
family of automorphisms  on $\r^{N+1}$ making $\Lc$ homogeneous of degree two 
\begin{equation} \label{dilation_invariance}
\Lc \circ \delta_r = r^2 \delta_r\circ \Lc \qquad \forall \  r>0.
\end{equation}
and whose explicitly expression is} given by
\begin{eqnarray} \label{dilations}
 \delta_r(x,t) & := & \delta_r(x^{(m_0)},x^{(m_1)},\dots,x^{(m_\kappa)},t)\notag\\
 & := &    (r^{2\vartheta +1}x^{(m_0)},r^{2\vartheta+2}x^{(m_1)},\dots,r^{2(\vartheta+\kappa)+1}x^{(m_\kappa)},r^2 t),
\end{eqnarray}
where
{%
$x^{(m_j)} \in \r^{m_j}$ for $j=0,\dots,\kappa$ and for $r>0$.} Throughout the sequel we indicate with $Q+2:=(2\vartheta+1) m_0 + (2\vartheta+2)m_1+\cdots (2\vartheta+2\kappa+1)m_\kappa +2$ the homogeneous dimension of~$\r^{N+1}$ with respect to $(\delta_r)_{r>0}$. The number $Q$ will be the {homogeneous dimension} of $\r^N$ with respect to the family of automorphisms $(D_r)_{r>0}$ given by
\begin{equation}\label{spatial dilation}
D_r : \r^N \to \r^N, \quad	D_r(x):=  \big(r^{2\vartheta +1}x^{(m_0)},\dots,r^{2(\vartheta+\kappa)+1}x^{(m_\kappa)}\big).
\end{equation}
{Throughout the paper, we will denote with $|\cdot|$ the Euclidean norm on $\r^N$, $\r^{m_j}$ (for $j=0,\dots,\kappa$) or $\r$. For any $x \in \r^N$ we denote with
\begin{equation}\label{def:normC}
|x|_{C}^2 := \frac{1}{4}\langle C^{-1}(-1,0) x,x\rangle.
\end{equation}
Denoting with $4\sigma_C^2$ the smallest eigenvalue of the positive definite matrix
\[
e^{-B^T}C^{-1}(-1,0)e^{-B}\,,
\]
we have 
\begin{equation}\label{eq:propC}
    |e^{-B}x|_{C}^2 \geq \sigma_C^2|x|^2.
\end{equation}
Moreover, we recall that a homogeneous norm $\|\cdot\|: \r^{N}\to \r_+$ is a $D_r$-homogeneous function of degree $1$ defined as follows
\[
\|x\|:= \sum_{j=0}^{\kappa}\left|x^{(m_j)}\right|^\frac{1}{2(j + \vartheta)+1}.
\]
We call homogeneous cylinder of radius $r>0$ and centered in $z_{\rm o}=(x_{\rm o},0)$ the set
\begin{equation}\label{eq:cylinder-homogeneous}
\mathcal{Q}_r(z_{\rm o}) := \big\{z=(x,t) \in \r^{N+1}: \|x-e^{-tB}x_{\rm o}\| < r, \ -r^{2} < t \leq 0 \big\}.
\end{equation}
The norms $\left\|\cdot\right\|$ and $\left|\cdot\right|$ can be compared as follows 
\begin{eqnarray}\label{confrhomnonhom}
&& \sigma\min{\left\{\left|x\right|^\frac{1}{1+2\vartheta}, \left|x\right|^{\frac{1}{2(\kappa+\vartheta)+1}}\right\}}\notag\\
&& \qquad \quad \leq\left\|x\right\|\leq (\kappa+1)\max{\left\{\left|x\right|^\frac{1}{1+2\vartheta}, \left|x\right|^{\frac{1}{2(\kappa+\vartheta)+1}}\right\}}\qquad \forall\,x\in\r^N,
\end{eqnarray}
where $\sigma=\min_{|x|=1}{\left\|x\right\|}$.
Indeed, on one side we simply have
\[
\left\|x\right\|\leq \sum_{j=0}^\kappa\left|x\right|^{\frac{1}{2(j+\vartheta)+1}}\leq (\kappa+1)\max{\left\{\left|x\right|^\frac{1}{1+2\vartheta}, \left|x\right|^{\frac{1}{2(\kappa+\vartheta)+1}}\right\}}\quad \forall\,x\in\r^N.
\]
On the other hand, for any $x\neq 0$, we get
\[
\frac{\left\|x\right\|}{\min{\left\{\left|x\right|^\frac{1}{1+2\vartheta}, \left|x\right|^{\frac{1}{2(\kappa+\vartheta)+1}}\right\}}}\geq\sum_{j=0}^\kappa{\frac{\left|x^{(m_j)}\right|^{\frac{1}{2(j+\vartheta)+1}}}{\left|x\right|^{\frac{1}{2(j+\vartheta)+1}}}}=\sum_{j=0}^\kappa{\left|\left(\frac{x}{\left|x\right|}\right)^{(m_j)}\right|^{\frac{1}{2(j+\vartheta)+1}}}=\left\|\frac{x}{\left|x\right|}\right\|\geq\sigma.
\]
}

 Let us give the definition of {\it $\Lc$-cone}.
\begin{defn}
	\label{l-cone}
	For any $T>0$, $R>0$ and any compact subset $K$ of $\r^N$ with positive Lebesgue measure we call $\Lc$\textup{-cone} of vertex $z_{\rm o}:=(x_{\rm o},0)$, base $K$ and height $T$, the set
$$
\mathcal{C}(z_{\rm o}) :=\left\{(D_rx+x_{\rm o},-r^2T): x \in K, 0 \leq r \leq R\right\},
$$
where $D_r$ is defined in \eqref{spatial dilation}.

Given an open subset $U$ of~$\ \R^{N+1}$ and ${z_{\rm o}=(x_{\rm o},0)} \in \partial U$ we say that there exists an \textup{exterior cone} with vertex in $z_{\rm o}$ if there exists an $\Lc$-cone $\mathcal{C}(z_{\rm o})$ such that $\mathcal{C}(z_{\rm o}) \subseteq \r^{N+1} \setminus U$.
\end{defn}

We recall some known facts about the operator $\Lc$ in the case $\vartheta = 0$, which will be useful for the study of the case when $\vartheta \ge 1$. 

\medskip

\textit{If $\vartheta = 0$, specifically when
\[
 \Lc := \sum_{j=1}^{m} \partial_{x_j}^2  + \langle B x, \nabla \rangle -\partial_t,
\]
the following statements are equivalent to the condition {\rm \bf[H.1]}:
\begin{enumerate}[(i)]
	\item {\rm(}H\"ormander's condition{\rm)}{\rm:} 
	${\rm rank \, Lie} (X_{1}, \ldots, X_{m}, Y) (x,t) = N + 1$ for every $(x,t) \in \R^{N+1}$;  
	\item ${\rm Ker} (A)$ does not contain non-trivial subspaces which are invariant for $B^T$;
	\item $C(\tau, t) > 0$ for every $t>\tau$;
	\item {\rm(}Kalman's rank condition{\rm)}{\rm:}  
	${\rm rank} \, \left( A, B A, \ldots, B^{N-1} A \right) = N$.
\end{enumerate}
}
\medskip 

The equivalence between {\it (i)} and {\it (ii)} was first proved by H\"ormander in \cite{Horm67}. A detailed proof of the equivalence between {\it (i)}, {\it (ii)}, {\it (iii)}  and {\bf[H.1]} can be found in \cite{LP94} (see Proposition A.1, and Proposition 2.1). The equivalence between {\it (iii)} and {\it (iv)} was pointed out by Lunardi in \cite{Lun97}. We next prove that the above result also holds in the case $\vartheta \ge 1$.

\begin{prop} \label{prop-hypo}
	The following statements are equivalent to the condition {\rm \bf[H.1]}: 
\begin{enumerate}[(i)]
	\item  {\rm(}H\"ormander's condition{\rm)}{\rm:} 
	${\rm rank \, Lie} (X_{1}, \ldots, X_{m}, Y) (x,t) = N + 1$ for every $(x,t) \in \R^{N+1}$;  
	\item ${\rm Ker} (A)$ does not contain non-trivial subspaces which are invariant for $B^T$;
	\item $C(\tau, t)>0$ for every $t>\tau$;
	\item {\rm}Kalman's rank condition{\rm)}{\rm:} 
	${\rm rank} \, \left( A, B A, \ldots, B^{N-1} A \right) = N$.
\end{enumerate}
\end{prop}

\begin{proof}
As said above, the assertion is known to be true in the case $\vartheta=0$. Moreover, the constant $\vartheta$  does not appear in {\rm \bf[H.1]}, {\it (ii)} and {\it (iv)}, hence the equivalence between {\rm \bf[H.1]}, {\it (ii)} and {\it (iv)} trivially follows from the case $\vartheta = 0$. 

\medskip

We next prove that {\rm \bf[H.1]} is equivalent to {\it (ii)} for every $\vartheta \ge 1$. With this aim, we compare condition {\it (ii)} with $\vartheta = 0$ and condition {\it (ii)} with $\vartheta \ge 1$. In order to distinguish the two cases we set, for $j=1, \dots, m$,
\begin{equation*}
	\widetilde X^0_j := \partial_{x_j}, \qquad \widetilde X^k_j := [ \widetilde X^{k-1}_j, Y ], \qquad k = 1, \dots, \kappa. 
\end{equation*}
Moreover, we let
\begin{equation*} 
	V_k := \text{span} \left\{ \widetilde X^k_1, \dots, \widetilde X^k_m \right\}, \qquad k=0, \dots, \kappa.
\end{equation*}
and we set, for $j=1, \dots, m$,
\begin{equation*}
	X^0_j := t^\vartheta \partial_{x_j}, \qquad  X^k_j := [ X^{k-1}_j, Y ], \qquad k = 1, \dots, \kappa. 
\end{equation*}
A direct computation shows that
\begin{equation*} 
	[ t^\vartheta \partial_{x_j}, Y ] = t^\vartheta [  \partial_{x_j}, Y ] + \vartheta t^{\vartheta-1} \partial_{x_j},
	\qquad j=1, \dots, m,
\end{equation*}
that can be written as follows
\begin{equation*} 
	X_j^1 = t^\vartheta \widetilde X_j^1 + \vartheta t^{\vartheta-1} \widetilde X_j^0, \qquad j=1, \dots, m.
\end{equation*}
By iterating the same argument, we find
\begin{equation} \label{eq-X2}
	X_j^2 = t^\vartheta \widetilde X_j^2 + 2 \vartheta t^{\vartheta-1} \widetilde X_j^1 + 
	\vartheta (\vartheta-1) t^{\vartheta-2} \widetilde X_j^0, \qquad j=1, \dots, m.
\end{equation}
and, for $k = 3, \dots$ and $j=1, \dots, m$, 
\begin{equation} \label{eq-Xk}
	X_j^k = t^\vartheta \widetilde X_j^k + k \vartheta t^{\vartheta-1} \widetilde X_j^{k-1} + \dots + 
	\vartheta (\vartheta-1) \dots (\vartheta - k + 1)  t^{\vartheta-k} \widetilde X_j^0.
\end{equation}
Note that the last coefficient vanishes whenever $k > \vartheta$.
	
We are now ready to show that H\"ormander's condition {\it (ii)} is satisfied by the system of vector fields $\{ X^k_1, \dots, X^k_m, Y \}$ in the set $\{t \ne 0 \}$. Indeed, we easily see that, in this case,
\begin{equation*}
	\text{span} \left\{ \widetilde X^0_1, \dots, \widetilde X^0_m \right\} = \text{span} \left\{ X^0_1, \dots, X^0_m \right\}.
\end{equation*}
Moreover \eqref{eq-X2} implies that 
\begin{equation*}
	\text{span} \left\{ \widetilde X^0_1, \dots, \widetilde X^0_m, \widetilde X^1_1, \dots, \widetilde X^1_{m_1} \right\} = \text{span} \left\{ X^1_1, \dots, X^1_{m_1} \right\},
\end{equation*}
for every $t \ne 0$. By the same reason, using the above assertions and \eqref{eq-Xk}, we conclude that ${\rm Lie}\{ X^k_1, \dots, X^k_m, Y \}$ agrees with ${\rm Lie}\{ \widetilde X^k_1, \dots, \widetilde X^k_m, Y \}$ whenever $t \ne 0$.
	
We are left with the set $\{t=0 \}$. In this case we use \eqref{eq-Xk} with $k = \vartheta$ and we find 
\begin{equation*} 
	X_j^\vartheta (x,0) =  \vartheta ! \, \widetilde X_j^0, \qquad j=1, \dots, m,
\end{equation*}
for every $x \in \R^N$. This means that $\widetilde X_j^0$ belongs to ${\rm Lie}\{ X^k_1, \dots, X^k_m, Y \}$ computed at $t=0$. Hence, ${\rm Lie}\{ X^k_1, \dots, X^k_m, Y \}$ contains ${\rm Lie}\{ \widetilde X^k_1, \dots, \widetilde X^k_m, Y \}$ and H\"ormander's condition \eqref{e-Hormander} is satisfied also in the set $\{t=0 \}$. This concludes the proof of the equivalence between  {\bf[H.1]} and {\it (ii)}.

\medskip

We next prove that {\it (ii)} is equivalent to {\it (iii)}. We follow H\"ormander's argument. We first note that the matrix $A e^{-sB^T}$ is non negative, for every $s \in \R$. Then $C(\tau,t) \ge 0$ whenever $t \ge \tau$. Moreover, the function $t \mapsto \langle C(\tau,t) \xi, \xi \rangle$ is non-decreasing. We claim that the following assertions are equivalent:
\begin{enumerate}
 \item there exists a $t_{\rm o} > \tau$ such that $\langle C(\tau,t_{\rm o}) \xi, \xi \rangle = 0$;
 \item $\langle C(\tau,t) \xi, \xi \rangle = 0$ fo every $t > \tau$;
 \item $A (B^T)^k \xi = 0$, for every non-negative integer $k$.
\end{enumerate}
We first prove that $1.$ implies $3.$ Assume that there exists a $t_{\rm o} > \tau$ and a vector $\xi \in \R^N$ such that $\langle C(\tau,t_{\rm o}) \xi, \xi \rangle = 0$. Then $\langle C(\tau,t) \xi, \xi \rangle = 0$ for every $t \in [\tau, t_{\rm o}]$. From the definition of $C(\tau,t)$ \eqref{c}, it follows that
\begin{equation*}
 s^{2\vartheta} \langle A e^{-sB^T} \xi, e^{-sB^T} \xi \rangle=0, \quad \text{for every} \quad s \in [\tau, t_{\rm o}],
\end{equation*}
then 
\begin{equation*}
 \bigg(\sum_{k=0}^{+ \infty} \frac{(-1)^k}{k!}s^{k + 2\vartheta}  A (B^T)^k \bigg) \xi = 0, \quad \text{for every} \quad s \in [\tau, t_{\rm o}],
\end{equation*}
which implies the assertion $3.$ The implications $3. \Rightarrow 2. \Rightarrow 1.$ are trivial and are omitted.  

The proof of the equivalence between {\it (ii)} and {\it (iii)} is a direct consequence of the fact that
\[
 V := \left\{ \xi \in \R^N : A (B^T)^k \xi = 0 \ \text{fo every non-negative integer} \ k \right\}
\]
is the greatest subspace of ${\rm Ker} (A)$ which is $B^T$--invariant. 
This completes the proof of Proposition \ref{prop-hypo}. 
\end{proof}

We emphasize that the condition {\it (iii)} of Proposition \ref{prop-hypo} is very important in our setting. Indeed, by using the Fourier transform we find the explicit expression of the fundamental solution of $\Lc$. Indeed, for every $z = (x,t), \zeta = (\xi,\tau) \in \r^{N+1}$ 
we have
\begin{equation} \label{fund-sol}
	\Gamma(z;\zeta):= 
    \begin{cases}
        \frac{(4\pi)^{-N/2}}{\sqrt{\textup{det}C(\tau,t)}}
	e^{-\frac{1}{4}\langle C^{-1}(\tau,t)( x- e^{-(t-\tau)B}\xi ) , x- e^{-(t-\tau)B}\xi \rangle} & \quad \text{if} \ t > \tau\\
    0 & \quad \text{if}\ t \le \tau.
        \end{cases}
\end{equation}
The expression \eqref{fund-sol} was first obtained by Kuptsov under a condition equivalent to {\it (iv)}, and used by Montanari in \cite{Mon96}. { We also recall the scaling property of the fundamental solution $\Gamma$ with respect to the automorphism \eqref{dilations}; see \cite[Lemma 2.1]{Mon96}.

\begin{lemma}\label{lemma:dilation-fundamental}
The following properties of the fundamental solution $\Gamma$ in \eqref{fund-sol} hold true:
\begin{enumerate}[(i)]
 \item   For any $(x,t),(\xi,\tau) \in \mathbb{R}^{N+1}$ 
         \[
         \Gamma(x,t;\xi,\tau) = \Gamma(x-e^{-(t-\tau)B}y,t;\xi-y,\tau) \quad \forall y \in \r^N.
         \]
 \item   For any $(x,t),(\xi,\tau) \in \mathbb{R}^{N+1}$ and any $r>0$ it holds
    \[
   \Gamma(x,t;\xi,\tau) =  r^{Q}\Gamma(D_r x, r^2 t; D_r \xi, r^2 \tau).
    \]
\end{enumerate}
\end{lemma}
The following properties of the fundamental solution $\Gamma$ in will the a key tool in the subsequent proof of the sufficient condition in Theorem \ref{main result2}.

We start by recalling the following identity, whose proof can be found in \cite[Remark 2.1]{LP94}
\begin{equation}\label{eq:commutator}
e^{-r^2sB}D_r=D_r e^{-sB}\qquad \forall r>0,\, \forall s\in \r.
\end{equation}
We will need the following lemma.
\begin{lemma}\label{matinv}
For $0>t>\tau$ we have the following matrix inequality
$$e^{-tB^T}C^{-1}(\tau,t)e^{-tB}\geq C^{-1}(\tau,0).$$
\end{lemma}
\begin{proof}
Let us begin noticing that for symmetric positive definite matrices we have
\[
M_1\leq M_2\quad \Rightarrow\quad M_1^{-1}\geq M_2^{-1}
\]
(see \cite[Corollary 7.7.4]{HJ90}) and recalling that $(e^{-tB})=e^{tB}$, it is enough to show that
\begin{equation}\label{matrixclaim}
e^{t B}C(\tau,t)e^{t B^T}\leq C(\tau,0).
\end{equation}
From the very definition of the matrix $C$ we get
\begin{eqnarray*}
e^{t B}C(\tau,t)e^{t B^T}&=& e^{tB}\left(\int_{0}^{t-\tau}{(t-s)^{2\vartheta}e^{-sB}Ae^{-sB^T}}\,\rm{d}s\right)e^{tB^T}\\
& = & \int_{0}^{t-\tau}{(t-s)^{2\vartheta} e^{(t-s)B}Ae^{(t-s)B^T}}\,\rm{d}s\\
&=&\int_{-t}^{-\tau}\sigma^{2\vartheta}{e^{-\sigma B}Ae^{-\sigma B^T}}\,\rm{d}\sigma.
\end{eqnarray*}
Since $-\tau>-t>0$ and $A$ is nonnegative definite, we have 
\[
\int_{-t}^{-\tau}{\sigma^{2\vartheta}e^{-\sigma B}Ae^{-\sigma B^T}}\,\rm{d}\sigma\leq \int_{0}^{-\tau}{\sigma^{2\vartheta}e^{-\sigma B}A e^{-\sigma B^T}}\,\rm{d}\sigma = C(\tau,0)
\]
which proves \eqref{matrixclaim} and the lemma.
\end{proof}

Let us now prove an estimate of the ratio $\frac{\Gamma(z,\zeta)}{\Gamma(z_{\rm o},\zeta)}$, for $z_{\rm o}=(x_{\rm o},0)$, $z=(x,t)$ and $\zeta=(\xi,\tau)$ with $0>t>\tau$. Let us denote with
\begin{equation}\label{eq:notation}
\begin{split}
\mu & :=\frac{-t}{-\tau}\in(0,1),\\
 {M}(z_{\rm o},z) & :=\left|D_{\frac{1}{\sqrt{-t}}}(x-e^{-tB}x_{\rm o})\right|,\\
 \text{and}\quad {M}(z_{\rm o},\zeta)& :=\left|D_{\frac{1}{\sqrt{-\tau}}}(\xi-e^{-\tau B}x_{\rm o})\right|
 \end{split}
\end{equation}

\begin{lemma}\label{rapporto}
Fix $z_{\rm o}=(x_{\rm o},0) \in \r^{N+1}$. There exists a positive constant $c$ such that, for any $z=(x,t),\zeta=(\xi,\tau)$ with $0>t>\tau$ and $\mu\leq\min{\{\frac{1}{2},\frac{\sigma^2}{(\kappa+1)^2}\}}$, we have
\[
\frac{\Gamma(z,\zeta)}{\Gamma(z_{\rm o},\zeta)}\leq\left(\frac{1}{1-\mu}\right)^{\frac{Q}{2}}e^{c\sqrt{\mu}M(z_{\rm o},z)M(z_{\rm o},\zeta)}.
\]
where $\mu$ and $M(\cdot)$ are both defined in \eqref{eq:notation}.
\end{lemma}
\begin{proof} 
Applying the  transformation in Lemma \ref{lemma:dilation-fundamental} we obtain that 
\begin{eqnarray*}
  \Gamma(z,\zeta)  & = & (t-\tau)^{-\frac{Q}{2}} \Gamma(D_{\frac{1}{\sqrt{t-\tau}}}(x-e^{-(t-\tau)B}\xi),\frac{t}{t-\tau};0,\frac{\tau}{t-\tau})\\
  & = & \frac{c(t-\tau)^{-\frac{Q}{2}}}{\sqrt{\det C(\frac{\tau}{t-\tau},\frac{t}{t-\tau})}}e^{-\frac{1}{4}\langle C^{-1}(\tau,t)(x-e^{-(t-\tau)B}\xi),x-e^{-(t-\tau)B}\xi\rangle}\\
   & = & \frac{c(t-\tau)^{-\frac{Q}{2}}}{\sqrt{\det C(-\frac{1}{1-\mu},-\frac{\mu}{1-\mu})}}e^{-\frac{1}{4}\langle C^{-1}(\tau,t)(x-e^{-(t-\tau)B}\xi),x-e^{-(t-\tau)B}\xi\rangle}\,,
\end{eqnarray*}
and, similarly, we have
\begin{eqnarray*}
    \Gamma(z_{\rm o},\zeta) &=&\frac{c(-\tau)^{-\frac{Q}{2}}}{\sqrt{\det C(-1,0)}}e^{-\left|D_{\frac{1}{\sqrt{-\tau}}}(x_{\rm o}-e^{\tau B}\xi)\right|_C^2}\,,
\end{eqnarray*}
recalling the definition of $|\cdot|_C$ in \eqref{def:normC}. 

Then, since 
\[
\frac{(t-\tau)^{-\frac{Q}{2}}\sqrt{\det C(-1,0)}}{(-\tau)^{-\frac{Q}{2}}\sqrt{\det C(-\frac{1}{1-\mu},-\frac{\mu}{1-\mu})}} = \left(\frac{1}{1-\mu}\right)^{\frac{Q}{2}}\,,
\] 
the only thing we need to control is the exponential term. For this, start noticing that
\begin{eqnarray}\label{eq:conto}
    x-e^{-(t-\tau)B}\xi & = & x - e^{-(t-\tau)B}(\xi -e^{-\tau B}x_{\rm o}+ e^{-\tau B}x_{\rm o})\notag\\
    & = & x - e^{-(t-\tau)B}e^{-\tau B}x_{\rm o} - e^{-(t-\tau)B}(\xi -e^{-\tau B}x_{\rm o})\notag\\
    & = &  x - e^{-tB}x_{\rm o}- e^{-(t-\tau)B}(\xi -e^{-\tau B}x_{\rm o})\,,
\end{eqnarray}
where we have used that $(e^{-tB})^{-1}=e^{t B}$. Then, by  \eqref{eq:conto} we obtain
\begin{eqnarray}\label{exponent}
&&-\frac{1}{4}\langle C^{-1}(\tau,t)(x-e^{-(t-\tau)B}\xi),x-e^{-(t-\tau)B}\xi\rangle\notag\\
&& \quad = -\frac{1}{4}\langle C^{-1}(\tau,t)(x-e^{-tB}x_{\rm o}),x-e^{-tB}x_{\rm o}\rangle\notag\\
&& \qquad -\frac{1}{4}\langle C^{-1}(\tau,t)e^{-(t-\tau)B}(\xi -e^{-\tau B}x_{\rm o}),e^{-(t-\tau)B}(\xi -e^{-\tau B}x_{\rm o})\rangle\nonumber\\
&& \qquad + \frac{1}{2}\left(\langle C^{-1}(\tau,t)(x-e^{-tB}x_{\rm o}),x-e^{-tB}x_{\rm o}\rangle\right)^\frac{1}{2}\notag\\
&& \qquad \times \left( \langle C^{-1}(\tau,t)e^{-(t-\tau)B}(\xi -e^{-\tau B}x_{\rm o}),e^{-(t-\tau)B}(\xi -e^{-\tau B}x_{\rm o})\rangle\right)^\frac{1}{2}\notag\\ 
&& \quad \leq -\frac{1}{4}\langle C^{-1}(\tau,t)e^{-(t-\tau)B}(\xi -e^{-\tau B}x_{\rm o}),e^{-(t-\tau)B}(\xi -e^{-\tau B}x_{\rm o})\rangle\\
&& \qquad + \frac{1}{2}\left(\langle C^{-1}(\tau,t)(x-e^{-tB}x_{\rm o}),x-e^{-tB}x_{\rm o}\rangle\right)^\frac{1}{2}\notag\\
&& \qquad \times \left( \langle C^{-1}(\tau,t)e^{-(t-\tau)B}(\xi -e^{-\tau B}x_{\rm o}),e^{-(t-\tau)B}(\xi -e^{-\tau B}x_{\rm o})\rangle\right)^\frac{1}{2} \,, \notag
\end{eqnarray}
since $C^{-1}(\cdot,\cdot)$ is a positive definite matrix.  Moreover, Lemma \ref{matinv} yields that for any $y \in \r^N$
 $$
 \left\langle C^{-1}(\tau,0)e^{\tau B}y,e^{\tau B}y\right\rangle-\left\langle C^{-1}(\tau,t) e^{-(t-\tau)B}y,e^{-(t-\tau)B}y\right\rangle\leq 0.
 $$
 By using this and \eqref{exponent} we get
\begin{eqnarray}\label{eq:exponent2}
&&\left|D_{\frac{1}{\sqrt{-\tau}}}\left(x_{\rm o}-e^{\tau B}\xi\right)\right|_C^2-\frac{1}{4}\langle C^{-1}(\tau,t)(x-e^{-(t-\tau)B}\xi),x-e^{-(t-\tau)B}\xi\rangle\notag\\
 && \quad \leq  \frac{1}{4}\langle C^{-1}(\tau,0)e^{\tau B}\left(\xi - e^{-\tau B}x_{\rm o}\right),e^{\tau B}\left(\xi - e^{-\tau B}x_{\rm o}\right)\rangle \notag\\
 && \qquad -\frac{1}{4}\langle C^{-1}(\tau,t)e^{-(t-\tau)B}(\xi -e^{-\tau B}x_{\rm o}),e^{-(t-\tau)B}(\xi -e^{-\tau B}x_{\rm o})\notag\\
&& \qquad + \frac{1}{2}\left(\langle C^{-1}(\tau,t)(x-e^{-tB}x_{\rm o}),x-e^{-tB}x_{\rm o}\rangle\right)^\frac{1}{2}\notag\\
&& \qquad \times \left( \langle C^{-1}(\tau,t)e^{-(t-\tau)B}(\xi -e^{-\tau B}x_{\rm o}),e^{-(t-\tau)B}(\xi -e^{-\tau B}x_{\rm o})\rangle\right)^\frac{1}{2} \notag\\
&& \quad \leq \frac{1}{2}\left(\langle C^{-1}(\tau,t)(x-e^{-tB}x_{\rm o}),x-e^{-tB}x_{\rm o}\rangle\right)^\frac{1}{2}\\
&& \qquad \times \left( \langle C^{-1}(\tau,t)e^{-(t-\tau)B}(\xi -e^{-\tau B}x_{\rm o}),e^{-(t-\tau)B}(\xi -e^{-\tau B}x_{\rm o})\rangle\right)^\frac{1}{2}\notag
\end{eqnarray}
We are going to bound all the above terms in \eqref{eq:exponent2} separately. We first have
\begin{eqnarray*}
&& \langle C^{-1}(\tau,t)(x-e^{-tB}x_{\rm o}),x-e^{-tB}x_{\rm o}\rangle\\*[0.5ex]
&& \quad =\left\langle C^{-1}\left(-\frac{1}{\mu},-1\right)D_{\frac{1}{\sqrt{-t}}}(x-e^{-tB}x_{\rm o}),D_{\frac{1}{\sqrt{-t}}}(x-e^{-tB}x_{\rm o})\right\rangle.
\end{eqnarray*}
Now, let us denote with $\left\|A\right\|$ the operator norm of a matrix $A$ (i.e. its biggest eigenvalue for symmetric matrices). By \eqref{confrhomnonhom}, for any vector $v$ with $\left|v\right|=1$ we get
\begin{eqnarray*}
\min{\left\{\left|D_{\sqrt{\mu}}v\right|^\frac{1}{1+2\vartheta},\left|D_{\sqrt{\mu}}v\right|^{\frac{1}{2(\kappa+\vartheta)+1}}\right\}}\leq\frac{1}{\sigma}\sqrt{\mu}\left\|v\right\| & \leq & \frac{\kappa+1}{\sigma}\sqrt{\mu}\max{\left\{\left|v\right|^\frac{1}{2\vartheta+1},\left|v\right|^{\frac{1}{2(\kappa+\vartheta)+1}}\right\}}\\
& = & \frac{\kappa+1}{\sigma}\sqrt{\mu}.
\end{eqnarray*}

From $\mu\leq \frac{\sigma^2}{(\kappa+1)^2}$ we then deduce $\left|D_{\sqrt{\mu}}v\right|\leq \left(\frac{\kappa+1}{\sigma}\right)^{2\vartheta +1}\mu^{\frac{1}{2}(2\vartheta +1)} \leq \left(\frac{\kappa+1}{\sigma}\right)^{2\vartheta +1}\sqrt{\mu} $ since $\mu \in (0,1)$. Hence, since by definition $\mu$ is also less than $\frac{1}{2}$,  for any $|v|=1$ we obtain that
\begin{eqnarray*}
\left\langle C^{-1}\left(-\frac{1}{\mu},-1\right)v,v \right\rangle&=&\left\langle C^{-1}(-1,-\mu)D_{\sqrt{\mu}}v,D_{\sqrt{\mu}}v\right\rangle\\
& \leq &  \left\|C^{-1}(-1,-\mu)\right\|\left|D_{\sqrt{\mu}}v\right|^2\\
&\leq& \left(\frac{\kappa+1}{\sigma}\right)^{4\vartheta +2}\left\|C^{-1}\left(-1,-\mu\right)\right\|\mu\\
& \leq &  \left(\frac{\kappa+1}{\sigma}\right)^{4\vartheta +2}\left\|C^{-1}\left(-1,-\frac{1}{2}\right)\right\|\mu\,,
\end{eqnarray*}
which gives recalling \eqref{eq:notation}
\begin{eqnarray*}
&& \left\langle C^{-1}(\tau,t)(x-e^{-tB}x_{\rm o}),x-e^{-tB}x_{\rm o}\right\rangle\notag\\
&& \qquad\quad \leq \left(\frac{\kappa+1}{\sigma}\right)^{4\vartheta +2}\left\|C^{-1}\left(-1,-\frac{1}{2}\right)\right\|\mu M(z_{\rm o},z)^2.
\end{eqnarray*}
On the other hand, by the commutation property \eqref{eq:commutator}, we get
\begin{eqnarray*}
&& \left\langle C^{-1}(\tau,t)e^{-(t-\tau)B}(\xi -e^{-\tau B}x_{\rm o}),e^{-(t-\tau)B}(\xi -e^{-\tau B}x_{\rm o})\right\rangle\notag\\
&& \quad =  \left\langle C^{-1}(-1,-\mu)D_{\frac{1}{\sqrt{-\tau}}}e^{-(t-\tau)B}(\xi -e^{-\tau B}x_{\rm o}),D_{\frac{1}{\sqrt{-\tau}}}e^{-(t-\tau)B}(\xi -e^{-\tau B}x_{\rm o})\right\rangle\\
&& \quad \leq\left\|C^{-1}(-1,-\mu)\right\|\left|D_{\frac{1}{\sqrt{-\tau}}}e^{-(t-\tau)B}(\xi -e^{-\tau B}x_{\rm o})\right|^2\notag\\
&& \quad = \left\|C^{-1}(-1,-\mu)\right\|\left|e^{-(1-\mu)B}D_{\frac{1}{\sqrt{-\tau}}}(\xi -e^{-\tau B}x_{\rm o})\right|^2\\
&& \quad \leq \left\|C^{-1}(-1,-\mu)\right\|\left\|e^{-(1-\mu)(B+B^T)}\right\|\left|D_{\frac{1}{\sqrt{-\tau}}}(\xi -e^{-\tau B}x_{\rm o})\right|^2\\
&& \quad \leq \left\|C^{-1}\left(-1,-\frac{1}{2}\right)\right\|\left\|e^{-(1-\mu)(B+B^T)}\right\|\left|D_{\frac{1}{\sqrt{-\tau}}}(\xi -e^{-\tau B}x_{\rm o})\right|^2.
\end{eqnarray*}
Since $0<\mu\leq\frac{1}{2}$, the term $\left\|e^{-(1-\mu)(B+B^T)}\right\|$ is bounded from above by a universal constant $c_0^2$. Thus we have
\begin{eqnarray*}
&& \left\langle C^{-1}(\tau,t)e^{-(t-\tau)B}(\xi -e^{-\tau B}x_{\rm o}),e^{-(t-\tau)B}(\xi -e^{-\tau B}x_{\rm o})\right\rangle\notag\\
&& \qquad\quad \leq c_0^2 \left\|C^{-1}\left(-1,-\frac{1}{2}\right)\right\|M(z_{\rm o},\xi)^2.
\end{eqnarray*}
Therefore
\[
\frac{\Gamma(z,\zeta)}{\Gamma(z_{\rm o},\zeta)}\leq\left(\frac{1}{1-\mu}\right)^{\frac{Q}{2}}e^{c\sqrt{\mu} M(z_{\rm o},z) M(z_{\rm o},\xi)}\,,
\]
which gives the thesis.
\end{proof}
}
Now, we are now in position to introduce the \textit{cylindrical} sets with basis at $z_{\rm o}=(x_{\rm o},t_{\rm o}) \in \r^{N+1}$ previously used in \cite{Mon96}. For every $T > t_{\rm o}$ and $r>0$, we let
\begin{equation} \label{Kolm cylinder}
	{Q}_{r,T}(z_{\rm o})  :=  \left\{z \in \r^{N+1} : t_{\rm o} < t< T, \ \left| D_{\frac{1}{\sqrt{r}}}(e^{t B}x-e^{t_{\rm o}B}x_{\rm o})\right|<1\right\},
\end{equation}
and we denote by $\partial_P{Q}_{r,T}(z_{\rm o})$ its parabolic boundary
\begin{equation*}
\partial_P{Q}_{r,T}(z_{\rm o})  =  \partial {Q}_{r,T}(z_{\rm o}) \setminus \left\{z=(x,T) \in \r^{N+1} : \ \left| D_{\frac{1}{\sqrt{r}}}(e^{T B}x-e^{t_{\rm o}B}x_{\rm o})\right|<1\right\}.
\end{equation*}
Here $\partial {Q}_{r,T}(z_{\rm o})$ is the topological boundary of ${Q}_{r,T}(z_{\rm o})$.
It has been shown by Montanari, in \cite{Mon96}, that, for every ${Q}_{r,T}(z_{\rm o})$ and for every $\p \in {C}(\partial{Q}_{r,T}(z_{\rm o}))$, there exists a unique solution $u \in {C}^{\infty}({Q}_{r,T}(z_{\rm o}))$  to the problem
\begin{equation}
	\label{pbm on cylinder}
	\begin{cases}
		\Lc u =0 \quad & \text{in} \ {Q}_{r,T}(z_{\rm o}),\\
		u =\p \quad & \text{in} \ \partial_P{Q}_{r,T}(z_{\rm o}).
	\end{cases}
\end{equation}
Moreover, again in \cite[Theorem 3.1]{Mon96} a Harnack inequality for positive solution to $\Lc u =0$ has been  proved.

We introduce some further notation. For every $\beta \in \r, 0<\alpha<\gamma<1$ and $\nu \in (0,\nu_{\rm o})$, with $\nu_{\rm o}>0$ depending on $\alpha$ and of the coefficients of the matrix $B$, and for every $\xi \in \r^N$ let us define the following sets
\begin{align*}
	Q^+&:= {Q}_{\nu r,(\beta +1)r}(\xi,(\beta -1)r) \cap \{\beta -\gamma \leq t/r \leq \beta - \alpha\},\notag \\
	Q^-  &:= {Q}_{\nu r,(\beta +1)r}(\xi,(\beta -1)r) \cap \{\beta + \alpha \leq t/r \leq \beta + \gamma\}\notag.
\end{align*}
We state the following Harnack inequality.
\begin{thm}	\label{harnack} 
	There exists a non-negative constant $c\equiv c(\alpha,\gamma,\beta,\nu)<\infty$ such that for all $r>0$ 
	\[
		\max_{\overline{Q^-}}u \leq c\,\min_{\overline{Q^+}}u,
	\]
	for all non-negative $u \in {C}^\infty(\overline{Q}_{\nu r,(\beta +1)r}(\xi,(\beta -1)r)$ satisfying 
    \[
    \Lc u = 0 \quad \text{in} \ {Q}_{\nu r,(\beta +1)r}(\xi,(\beta -1)r).
    \]
\end{thm}

\section{Review of Abstract Potential Theory}\label{sec:3}
We begin recalling some definitions and results from Potential Theory. We adopt the notation of the monograph \cite{CC72} by Constantinescu and Cornea. Let us indicate with $(\Ec, d_{\Ec}$) a metric space, locally connected and locally compact. Moreover, denoting with $\tau_{\Ec}$ the topology generated by the metric $d_{\Ec}$ on $\Ec$, we assume that $(\Ec,\tau_{\Ec})$ has a countable basis of open sets.

\begin{defn}
	Suppose we are given, for every open set $U \in \tau_{\Ec}$, a family $\Hd(U)$ of extended real valued functions $u : U \rightarrow [- \infty, \infty]$. We say that the map
	$$
    \Hd: U \longmapsto \Hd(U),
	$$
	is a \textup{sheaf of functions} on $\Ec$ if the following properties hold:
	\begin{enumerate}[(i)]
		\item If $U_1 , U_2 \in \tau_{\Ec}$  with $ U_1 \subseteq U_2$ and  $u \in \Hd(U_2)$  then $u_{|U_1} \in \Hd(U_1)$.
		\item if $(U_{i})_{i \in {I}} \in \tau_{\Ec}$  and $u: \bigcup_{i\in {I}}U_{i} \rightarrow [-\infty, \infty]$ is such that~$u_{|U_{i}} \in \Hd(U_{i})$ for all $i \in {I}$, then~u $\in \Hd(\bigcup_{i \in {I}}U_{i})$.  
	\end{enumerate}
A sheaf of functions $\Hd$ on $\Ec$ will be called \textup{harmonic} if,
for every $U \in \tau_{\Ec}$, $\Hd(U)$ is a linear subspace of ${C}(U)$.
A sheaf of functions $\Ud$ on $\Ec$ will be said  \textup{hyperharmonic} if, for any  $U \in \tau_{\Ec}$, the family $\Ud(U)$ is a convex cone of lower semi-continuous, lower finite functions.
\end{defn}

Note that if $\Ud$ is a hyperharmonic sheaf on $\Ec$, then the map 
$$
\Hd_{\Ud}: U \longmapsto \Ud(U) \cap (-\Ud(U)) \quad \forall U \in \tau_{\Ec},
$$
is a harmonic sheaf on $\Ec$. 

Throughout the sequel we indicate with\ $\Hd$ (resp. $\Ud$) a harmonic (resp. hyperharmonic) sheaf on $\Ec$ and $\Hd_{\Ud}$-functions (resp. $\Ud$-functions) will be called {\it harmonic} (resp. {\it hyperharmonic}). Moreover, a function $u \in (-\Ud)$ will be called \textit{hypoharmonic}.

\vspace{2mm}
Let $U \subseteq \Ec$ be an open set and let $\p: U \rightarrow (-\infty,+\infty]$ be a lower semi-continuous function. Then, for any open set $V \subset U$, with compact closure and non-empty boundary, and for any non-negative
 Radon measure $\mu$ on $\partial V$ we define 
\begin{equation}
	\label{sec2 1}
\int_{\partial V} \p \, {\rm d}\mu : = \sup \Biggl\{\int_{\partial V} g \, {\rm d}\mu : g \in {C}(\partial V), \ g \leq \p \ \text{on} \ \partial V \Biggl\}.
\end{equation}
Since $\p$ is lower finite and $\partial V$ is compact, $\p$ is bounded from below on  $\partial V$. Hence the set on the righthand side in (\ref{sec2 1}) is not empty. {Thus,} we can give the following definition.

\begin{defn} \label{hsweeping}
Let $V\subset \Ec$ be 
open, with compact closure and non-empty boundary. Let us consider a family  $\mu^V=\{\mu^V_x\}_{x \in V}$ of non-negative Radon measures on $\partial V$. The family $\mu^V$ will be called a \textup{sweeping on $V$}. For any lower semi-continuous function $\p: \partial V \rightarrow (-\infty,+\infty]$ we will denote with $\mu^V_\p$ the function
\begin{eqnarray*}
	&& \mu^V_\p : V \to (-\infty,+\infty],\\
 && \qquad x  \longmapsto  \mu^V_\p (x) := \int_{\partial V} \p \, {\rm d}\mu^V_x.
\end{eqnarray*}
If $\Hd$ is a harmonic sheaf on $\Ec$, then the sweeping $\mu^V$ will be called $\Hd$\textup{-sweeping} if:
\begin{enumerate}[(i)]
\item $\forall \p \in {C}(\partial V)$ the function $\mu^V_\p$ is a $\Hd$-function;
\item for any $\Hd$-function $h$ defined on an open neighbourhood of $\overline{V}$ we have $\mu^V_h = h$ on $V$.
\end{enumerate}
\end{defn}

We will say that the family
\begin{equation}
	\label{sweeping system}
\Omega := \biggl\{\mu^{V_i}=\{\mu^{V_i}_x\}_{x \in V_i}: i \in {I}\biggl\},
\end{equation}
is a \textit{sweeping system} on $\Ec$ if $\{V_i : i \in {I}\}$ is a basis for $\Ec$ of relatively compact sets with non-empty boundary and for any $i \in {I}$ $\mu^{V_i}$ is a sweeping  on $V_i$. 
\vspace{2mm}

If $\Hd$ is a harmonic sheaf on $E$, then a sweeping system $\Omega$ is called $\Hd$-{\it sweeping system} on $E$ if $\mu^{V_i}$ is a $\Hd$-sweeping on $V_i$, for every $i \in I$. 

A hyperharmonic sheaf can be defined starting from a sweeping systems. Indeed, let  us consider on $\Ec$ a sweeping system $\Omega$ as defined in (\ref{sweeping system}) and give the following definition.

\begin{defn}\label{hyper sheaf generated by Omega}
Let $U \subseteq \Ec$. A lower semicontinuous function $u :U \rightarrow (-\infty,+\infty]$ will be said $\Omega$-\textup{hyperharmonic} if for any $i \in {I}$ such that $V_i \Subset U$ we have that $\mu^{V_i}_u \leq u$ on $V_i$, i.e.
\begin{equation*}
u(x) \geq \int_{\partial V_i}u \, {\rm d}\mu^{V_i}_x, \qquad \forall x \in V_i.
\end{equation*}
The function $u$ will be said \textup{locally $\Omega$-hyperharmonic} if there exists an open covering $\{W_\jmath\}_{\jmath \in J}$ of $U$ such that, $\forall \, \jmath \in J$, $u_{|W_\jmath}$ is $\Omega$-hyperharmonic on $W_\jmath$. 

Let $\Omega$ be a sweeping system on the space $\Ec$. We call \textup{hyperharmonic sheaf generated by $\Omega$} the map  $\Ud$ defined as follows
\[
\Ud: \tau_{\Ec} \ni U \longmapsto \Ud(U):=\{u : u \ \text{is locally $\Omega$-hyperharmonic on} \ U\}.
\]
\end{defn}

Given the hyperharmonic sheaf $\Ud$ generated by the sweeping system $\Omega$, we call \textup{harmonic sheaf generated by} $\Omega$ the harmonic sheaf given by
\[
\Hd_{\Ud}: U \longmapsto \Ud(U) \cap (-\Ud(U)) \quad \forall U \in \tau_{\Ec}.
\]

\subsection{Resolutive sets}
Throughout the rest of this section $\Ud$ will denote a given hyperharmonic sheaf  on the space $\Ec$. Let us give the following definition.
\begin{defn}\label{mp-set}
An open set $U \subseteq \Ec$ will be called a \textup{minimum principle set}, in short a \textup{MP-set}, if every $\Ud$-function $u$ which is non-negative outside the intersection with $U$ of a compact set $K \subseteq \Ec$ and
\[
\liminf_{x \rightarrow y}u(x)\geq 0 \quad \forall y \in \partial U,
\]
is non-negative on $U$.
\end{defn}
\begin{rem}{\rm
We point out that, if in the previous definition we are considering an open set $U$ with compact closure, we drop the condition that a $\Ud$-function $u$ is non-negative outside the intersection with $U$ of a compact set $K \subseteq \Ec$.}
\end{rem}

Let $\Ud$ be the hyperharmonic sheaf on $\Ec$, $U \subseteq \Ec$ be a {MP}-set and let $\p: \partial U \rightarrow [-\infty,+\infty]$. Let us consider the set 
\begin{multline*}
\overline{\Ud}^U_\p  :=  \biggl\{u \in \Ud(U):   \overline{\{u < 0\}} \ \text{is a compact, possibly empty, subset of $U$}\\
  \liminf\limits_{U \ni x \rightarrow y}u(x) \geq \p(y) \ \forall y \in \partial U \biggr\}.
\end{multline*}

The sets $\overline{\Ud}^U_\p$ and $\underline{\Ud}^U_\p= - \overline{\Ud}^U_{-\p}$ will be called respectively the set of \textit{upper-functions} and the set of \textit{lower-function}. We will call \textit{upper-solution} and \textit{lower-solution} the functions:
$$
\overline{H}^U_\p :=\inf \overline{\Ud}^U_\p, \qquad \underline{H}^U_\p := \sup \underline{\Ud}^U_\p.
$$
The next proposition is a straightforward consequence of the definition of upper and lower solution.

\begin{prop} 
Let $\p_1, \p_2 : \partial U \rightarrow \overline{\mathbb{R}}, \alpha \in \mathbb{R}, \alpha >0$. Then:
	\begin{enumerate}[(i)]
		\item $\p_1 \leq \p_2 \Rightarrow \overline{H}^U_{\p_1} \leq \overline{H}^U_{\p_2}, \underline{H}^U_{\p_1} \leq \underline{H}^U_{\p_2} ,$
		\item $\overline{H}^U_{\p_1 +\p_2} \leq \overline{H}^U_{\p_1} + \overline{H}^U_{\p_2}, \underline{H}^U_{\p_1 +\p_2} \geq \underline{H}^U_{\p_1} + \underline{H}^U_{\p_2}  $, whenever the sums are defined.
		\item $\underline{H}^U_{\alpha \p_1} = \alpha \underline{H}^U_{\p_1}, \overline{H}^U_{\alpha \p_1} = \alpha \overline{H}^U_{\p_1}, \overline{H}^U_{-\alpha \p_1} = -\alpha \underline{H}^U_{\p_2}$,
		\item $\p_1 \geq 0 \Rightarrow  \overline{H}^U_{\p_1} , \underline{H}^U_{\p_1} \geq 0$.
	\end{enumerate}
\end{prop}
Let us given now a crucial definition.
\begin{defn}
	 A function $\p: \partial U \rightarrow [-\infty, \infty]$ is called \textup{resolutive} if the functions $\overline{H}^U_\p, \underline{H}^U_\p$ are $\Hd_{\Ud}$-functions and coincide. In this case we set $H^U_\p := \overline{H}^U_\p =\underline{H}^U_\p$ and we say that $H^U_\p$ is the \textup{generalized solution in the sense of Perron-Weiner} (in short PW solution).
	 
	  An open set $U$ of $\Ec$, with non-empty boundary, is said to be a \textup{resolutive set} (with respect to $\Ud$) if every $\p \in {C}_c(\partial U)$ is resolutive.
\end{defn}If $U$ is a resolutive set, for any $x \in U$, the map
$$
{C}_c(\partial U) \ni \p \longmapsto H_{\p}^U (x) \in \r,
$$
is a linear and non-negative functional, hence by the Riesz Theorem, for every $x \in U$, there exists a suitable Radon measure $\mu^U_x$ on $\partial U$ such that
$$
H^U_\p(x)=\int_{\partial U} \p(y) \, {\rm d}\mu^U_x(y).
$$
The measure $ \mu_x ^U $ is called the $\Hd_{\Ud}$\textit{-harmonic measure} related to $U$ and $x$. Clearly the family $\mu^U := \{\mu^U_x\}_{x \in U}$ is a sweeping on $U$ and, if $\overline{U}$ is compact, the family $\mu^U$ is a $\Hd_{\Ud}$-sweeping on $U$.

\subsection{Harmonic spaces and $\mathfrak{P}$-harmonic spaces}
Let us begin defining a harmonic space.
\begin{defn}
	\label{harmonic space}
The couple $(\Ec,\Ud)$, where $\Ud$ is a hyperharmonic sheaf on $\Ec$, is called a \textup{harmonic space} if the following axioms are satisfied:
\begin{enumerate}[(i)]
\item {\rm (A1)(Positivity):} For every $x\in \Ec$ there exists a $\Hd_{\Ud}$-function, defined in a neighbourhood of $x$, that does not vanish at $x$.
\item {\rm (A2)(Bauer convergence property):} Let $\{u_n\}_{n\in \mathbb{N}}$ be a monotone increasing sequence of $\Hd_{\Ud}$-functions on an open set $U$ of $\Ec$. Then
$$
u:=\lim_{n\rightarrow +\infty}u_n, 
$$
is a $\Hd_{\Ud}$-function whenever it is locally bounded.
\item {\rm (A3)(Resolutivity):} The resolutive sets (with respect to $\Ud$) form a basis for the topology $\tau_{\Ec}$ on $\Ec$.
\item {\rm (A4)(Completeness):} A lower semi-continuous, lower finite function $u$ on an open set $U$ of $\Ec$ belongs to $\Ud(U)$ if, for any relatively compact with non-empty boundary resolutive set $V$ (with respect to $\Ud$) such that $\overline{V} \subset U$ , we have $\mu^V_u \leq u$ on $V$, that is
$$
u(x) \geq \int_{\partial V}u \, {\rm d}\mu^V_x, \qquad \forall x \in V,
$$
where $\mu^V$ is given by the sweeping constructed with the basis of resolutive sets.
\end{enumerate}
\end{defn}

\begin{rem}{\rm 
In the particular case the hyperharmonic sheaf $\Ud$ is generated by a sweeping system $\Omega$ (see Definition \textup{\ref{hyper sheaf generated by Omega}}), the axiom \textup{{\rm (A4)}} of Completeness, is trivially satisfied. }
\end{rem}

In our setting, by using the Harnack inequality for the non-negative solutions to $\Lc u = 0$ given in Theorem~\ref{harnack}, we will prove the following property which, in turn, implies the Bauer convergence property {\rm (A2)}.
\begin{enumerate}
\item[\it (iv)] {\rm (A2)'(Doob convergence property):} If $\{u_n\}_{n \in \mathbb{N}}$ is a monotone increasing sequence of $\Hd_{\Ud}$-functions on an open set $U \subset \Ec$ such that the set
$$
\biggl\{ x \in U | \, \sup_{n \in \mathbb{N}}u_n (x) < \infty \biggl\},
$$
is dense in $U$, then 
$$
u := \lim_{n \rightarrow \infty} u_n,
$$
is a $\Hd_{\Ud}$-function on $U$.
\end{enumerate}
Throughout the sequel we indicate with $(\Ec,\Ud)$ a harmonic space. 

\begin{defn}
A hyperharmonic function $u$ on a harmonic space $(\Ec,\Ud)$ is called \textup{superharmonic} if, for any relatively compact resolutive set $V$, the function $\mu^V_u$ is harmonic. A hypoharmonic function $u$ will be said \textup{subharmonic} if $-u$ is superharmonic.
\end{defn}
\begin{rem}\label{sup finite on a dense set}{\rm
Every superharmonic function $u$ is finite on a dense subset of its domain. Moreover, if the harmonic sheaf $\Hd_{\Ud}$ has the Doob convergence property \textup{{\rm (A2)'}}, then hyperharmonic functions, which are finite on a dense set, are superharmonic.
}
\end{rem}

\begin{defn}
 A non-negative superharmonic function $p$ for which any non-negative harmonic minorant vanishes identically is called a \textup{potential}.
\end{defn}
We refer to \cite{CC72} for some properties of superharmonic functions and potentials.  Let us give the following definition.

\begin{defn}\label{beta armonico} A harmonic space $(\Ec,\Ud)$ will be called $\mathfrak{P}$-\textup{harmonic space} if for any $x \in \Ec$ there exists a potential $p$ on $\Ec$ such that $p(x)>0$.
\end{defn}
The following result holds (see \cite[Proposition 2.3.2 ]{CC72}).
\begin{prop}
\label{thm 1}
Let $(\Ec,\Ud)$ be a harmonic space. The following conditions are equivalent:
\begin{enumerate}[(i)]
\item $\Ec$ is a $\mathfrak{P}$-harmonic space;
\item the set $\mathdutchcal{P}_c$ of finite continuous potentials on $\Ec$ such that any $p\in \mathdutchcal{P}_c$ is harmonic outside a compact set, separates the points of $\Ec$;
\item the set of non-negative superharmonic functions on $\Ec$ separates the points of $\Ec$;
\item for any relatively compact, resolutive set $V$ and for any $x \in V$, there exists a non-negative, finite, continuous superharmonic function $u$ on $\Ec$ such that 
$$
\int_{\partial V} u \, {\rm d}\mu^V_x < u(x).
$$
\end{enumerate}
\end{prop}

As a consequence of the proposition above we have the following corollary \cite[Corollary 2.3.3]{CC72}.
\begin{corol}\label{open set is mp-set}
Every open set of a $\mathfrak{P}$-harmonic space is an MP-set, according to Definition \textup{\ref{mp-set}}.
\end{corol}
$\mathfrak{P}$-harmonic spaces are really important since the following result holds true \cite[Theorem 2.4.2]{CC72}.
\begin{thm}
\label{thm 2}
Any open set of a $\mathfrak{P}$-harmonic space with non-empty boundary is resolutive.
\end{thm}
The consequence of Theorem \ref{thm 2} is that given an open set $U$ of a $\mathfrak{P}$-harmonic space the $\Hd_{\Ud}$-Dirichlet problem
\begin{equation}\label{dirichlet pt}
\begin{cases}
u \in \Hd_{\Ud}(U),\\
u=\p \quad \text{on}\ \partial U,\ \forall \p \in {C}_c(\partial U),
\end{cases}
\end{equation}
admits a solution $H^U_\p$ in the sense of {{Perron-Weiner-Brelot-Bauer}}.
\vspace{2mm}

In general, we cannot expect a good behaviour of $H^U_\p$ at the boundary points of $U$. In the following section we describe the conditions under which the boundary datum $\p$ in \eqref{dirichlet pt} is attained by the generalized solution $H^U_\p$.

\subsection{Boundary regularity}
Let us give the following definitions.
\begin{defn}
	 Let $(\Ec,\Ud)$ be a $\mathfrak{P}$-harmonic space and let $U$ be an open subset of $\Ec$ with non-empty boundary. A point $x_{\rm o} \in \partial U$ is said \textup{$\Hd_{\Ud}$-regular} if 
	$$
	\lim_{x \rightarrow x_{\rm o}}H^U_\p (x)= \p(x_{\rm o}), \qquad \forall \p \in {C}_c(\partial U).
	$$
	A point $x_{\rm o} \in \partial U$ which is not regular is called $\Hd_{\Ud}$-irregular.
\end{defn}

\begin{defn}
Let $(\Ec,\Ud)$ be a $\mathfrak{P}$-harmonic space and let $U$ be an open set of $\Ec$ with non-empty boundary, $x_{\rm o} \in \partial U$ and let $V$ be an open neighbourhood of $x_{\rm o}$. We say that a function $\omega \in \Ud(V \cap U)$ is a \textup{barrier at $x_{\rm o}$}  if:
\begin{enumerate}[(i)]
\item $\omega >0$ on $U \cap V$;
\item $\lim\limits_{x \rightarrow x_{\rm o}}\omega(x)=0.$
\end{enumerate}
\end{defn}
The first condition for a boundary point to be regular is having a barrier function \cite[Proposition 2.4.7]{CC72}.
\begin{prop}
Let $U$ be a resolutive subset of a $\mathfrak{P}$-harmonic space $(\Ec,\Ud)$. Then, any boundary point $x_{\rm o}$ which possesses a barrier is regular. 
\end{prop}


In order to state some geometrical characterization of regular point we need some further notation. The following notion was introduced by Brelot (\!\!\cite{Bre60}). 

\begin{defn}\label{balayage}
Let $(\Ec,\Ud)$ be a $\mathfrak{P}$-harmonic space. For any non-negative function $u$ on $\Ec$ and any subset $A$ of $\Ec$ denote
\[
\Phi_A^u :=\left\{ v \in \Ud(\Ec): v \geq 0 \ \text{on} \ \Ec \ \text{and} \ v \geq u \ \text{on} \ A\right\}.
\]
We call the \textup{reduit} of $u$ on $A$ the following function
\begin{equation*}
\reduit^u_A := \inf\big\{ v : v \in \Phi_A^u\big\}.
\end{equation*} 
We call \textup{balayage} of $u$ on $A$ the lower semi-continuous regularization of the reduit function of $u$ on $A$, that is
\begin{equation*}
\balayage^u_A (x) := \liminf_{y \rightarrow x} \reduit^u_A(y) \quad \forall x \in \Ec.
\end{equation*}
\end{defn}

We list  some useful properties of the balayage and the reduit function which will turn out to be helpful in the following parts of the paper.

\begin{prop}\label{subadd of balayage}
For any subsets $A$ and $B$ of $\Ec$ and for every non-negative function $u$ and~$v$ on $\Ec$ the following properties hold
	\begin{enumerate}[(i)]
\item $\reduit^u_A =u$ on~$A$;
\item if $A \subseteq B$ and~$u \leq v$ we have~$\reduit^u_A \leq \reduit^v_B$;
\item $\balayage^u_A = \reduit^u_A$ if~$A$ is open;
\item $\balayage^u_{A \cup B} +\balayage^u_{A \cap B} \leq \balayage^u_{A} + \balayage^u_{B}$, and $ \reduit^u_{A \cup B}+ \reduit^u_{A \cap B} \leq \reduit^u_{A}+\reduit^u_{B}$.
\end{enumerate}
\end{prop}
 For a proof of the last property in the proposition above we refer to~\cite[Theorem~4.2.2]{CC72}.

\begin{prop}
Let $(\Ec,\Ud)$ be a $\mathfrak{P}$-harmonic space and let $u$ be a non-negative superharmonic function on $\Ec$ and  $A\subset \Ec$. Then, $\reduit^u_A$ is harmonic on $\Ec \setminus \overline{A}$ and $\reduit^u_A$ and $\balayage^u_A$ coincide on $\Ec \setminus \overline{A}$.
\end{prop}

\begin{prop}
Let $(\Ec,\Ud)$ be a $\mathfrak{P}$-harmonic space. The balayage of any non-negative superharmonic function on $\Ec$, on any compact subset of $\Ec$, is a potential.
\end{prop}

For a proof of the previous propositions we refer to \cite[Proposition 5.3.1 and Proposition 5.3.5]{CC72}.
  
\begin{defn}\label{polar}
	Let $(\Ec,\Ud)$ be a $\mathfrak{P}$-harmonic space and let $U$ be an open set of $\Ec$. A set $P$ is said \textup{polar set} in $U$ if there exists a non-negative superharmonic function $p$ on $U$ which is equal $+\infty$ at least on $U \cap P$. In this case, we say that the function $p$ is \textup{associated to} $P$.	
\end{defn}
Even though most of the results stated below are proved in \cite{Bre60}, we give here their proofs, since our axiomatic setting is slightly different than the one adopted by the author of \cite{Bre60}.

\begin{prop}\label{polar1}
Let $(\Ec,\Ud)$ be a $\mathfrak{P}$-harmonic space, $P$ a polar set in $\Ec$ and $u$ a non-negative function on $\Ec$. Then, the reduit $\reduit^u_P$ is zero on a dense subset of $\Ec$. Moreover, $\balayage^u_P\equiv 0$.
\end{prop}
\begin{proof}
Suppose that $P$ is a polar set and consider its associated function $p$. Then, $p = +\infty$ on $P$. Choose any point $x_{\rm o}$ where $p(x_{\rm o})<+\infty$. We have that $\lambda p \geq u$ on $P$, for every $\lambda >0$, moreover $p\geq 0$ outside $P$. Then, 
\begin{equation}
\label{polar e1}
\lambda p \geq \reduit^u_P \quad \forall \lambda >0.
\end{equation}
Since \eqref{polar e1} holds also in $x_{\rm o}$ and $p(x_{\rm o})$ is finite, taking the infimum on $\lambda>0$ we get that $\reduit^u_P(x_{\rm o})=0$. By the previous argument we have that $\reduit^u_P$ is zero on every point in which $p$ is finite. Since $p$ is a superharmonic function, it is finite on a dense subset of $\Ec$ (see Remark \ref{sup finite on a dense set}). Then $\reduit^u_P =0$ on a dense subset of $\Ec$. Form this fact it follows that $\balayage^u_P \equiv 0$.
\end{proof}

As an immediate consequence of Proposition~\ref{polar1} and~\ref{subadd of balayage} we have
\begin{corol}\label{polar corol}
Let $(\Ec,\Ud)$ be a $\mathfrak{P}$-harmonic space. If $P$ is a polar set of $\Ec$, then $\balayage^u_{A \cup P} = \balayage^u_A$, for any subset $A$ of $\Ec$ and for any non-negative function $u$ on $\Ec$.
\end{corol}

The following definition will be used to give a further characterization of regular points.
\begin{defn}\label{thinnes}
Let $A$  be a subset of a $\mathfrak{P}$-harmonic space $(\Ec,\Ud)$ and let us consider a point $x_{\rm o} \not \in A$. We say that $A$ is \textup{thin} at $x_{\rm o}$ if either $x_{\rm o} \not \in \overline{A}$ or $x_{\rm o} \in \overline{A}$ and there exists a non-negative superharmonic function $u$ on $\Ec$ such that
\begin{equation*}
u(x_{\rm o})<\liminf_{A \ni x \rightarrow x_{\rm o}}u(x).
\end{equation*}
Let us consider a point $x_{\rm o} \in A$. We say that $A$ is \textup{thin} at its point $x_{\rm o}$ if $\{x_{\rm o}\}$ is a polar set in $\Ec$ (according to Definition \textup{\ref{polar}}) and $A \setminus \{x_{\rm o}\}$ is thin at $x_{\rm o}$.
\end{defn}
{
Let us remark that we will call a set $K \subset \Ec$ a $G_\delta$-set if $K$ is the countable intersection of open sets of $\Ec$. The following Proposition holds.
\begin{prop}\label{thin prop1}
Let $(\Ec,\Ud)$ be a $\mathfrak{P}$-harmonic space, let $A$ be any subset of $\Ec$, $x_{\rm o} \not \in A$ and let $w>0$ be a {superharmonic} function on $\Ec$, finite and continuous at $x_{\rm o}$. Then, $A$ is thin in $x_{\rm o}$ if and only if  there exists an open neighborhood $V$ of $x_{\rm o}$ such that
\begin{equation}\label{thin e1}
\reduit^w_{A \cap V} (x_{\rm o}) < w(x_{\rm o}) \quad \text{and} \quad \balayage^w_{A \cap V} (x_{\rm o}) < w(x_{\rm o}).
\end{equation}
\end{prop}
\begin{proof}
Let us begin noticing that if $A$ is thin at $x_{\rm o}$, then by \cite[Theorem 29]{Bre60}, the first condition in \eqref{thin e1} holds true. Then, the second one follows using the definition of balayage function as well as the continuity of $w$ in $x_{\rm o}$.

Let us prove the vice versa. Since $\reduit^w_{A \cap V} \geq \balayage^w_{A \cap V}$, it is enough to show that the second condition in \eqref{thin e1} implies that $A$ is thin in $x_{\rm o}$.  
First of all, let us note that $\{x_{\rm o}\}$ is a $G_\delta$-set, since in $\Ec$ the singleton $\{x_{\rm o }\}$ is the zero level set of the distance function which is a $G_\delta$-set.
With no loss of generality let us assume $x_{\rm o}  \in \overline{A}$, otherwise there is nothing to prove.
We  endow $\Ec$ of the {\it fine topology}, which is the coarsest topology on $\Ec$ which is finer than $\tau_{\Ec}$ and for which any hyperharmonic function on any open set is continuous. By \cite[Corollary 5.3.2]{CC72} since $\{x_{\rm o}\}$ is a $G_\delta$-set,  for any $\varepsilon>0$ such that
\begin{equation}\label{e0}
    \balayage^w_{A \cap V}(x_{\rm o}) + \varepsilon < 
    w(x_{\rm o})\,,
\end{equation}
there exists a positive superharmonic function $u$ on $\Ec$ such that $u \equiv w$ on $A$ and
\begin{equation}\label{e10}
    u(x_{\rm o}) \leq \balayage^w_{A \cap V}(x_{\rm o}) + \varepsilon.
\end{equation}
 Moreover, $\reduit^w_{A \cap V}(x_{\rm o}) = \balayage^w_{A \cap V}(x_{\rm o})$.
 Then, by the continuity of $w$ in $x_{\rm o}$, combined with \eqref{e0} and \eqref{e10}, we have that
 \[
 u(x_{\rm o}) < w(x_{\rm o}) = \lim_{A \ni x \to x_{\rm o}}w(x) \leq \liminf_{A \ni x \to x_{\rm o}}u(x).
 \]
 Thus $A$ is thin at $\{x_{\rm o}\}$.
\end{proof}
}
The following theorem holds true; see \cite[Theorem 6.3.3]{CC72}.

\begin{thm}\label{thin thm1}
Let $U$ be an open subset of a $\mathfrak{P}$-harmonic space $(E,\Ud)$ and $x_{\rm o} \in \partial U$. Hence, $x_{\rm o}$ is $\Hd_{\Ud}$-regular if and only if $E\setminus U$ is not thin at $x_{\rm o}$. 
\end{thm}

We conclude this section recalling a  well known result in potential theory that links the regularity of the boundary points of an open set $U$ with the balayage on the complementary of $\Ec \setminus U$. The forthcoming result is proven in \cite[Theorem 14]{NS84}; see also \cite[Theorem 4.6]{LU10}.

\begin{thm}
	\label{sec2 lemma}
	Let $(\Ec,\Ud)$ be a $\mathfrak{P}$-harmonic space, $U$ be an open subset of $\Ec$ and $x_{\rm o} \in \partial U$ such that $\{x_{\rm o}\}$ is a polar set in $\Ec$, according to Definition \textup{\ref{polar}}. Then, $x_{\rm o}$ is a $\Hd_{\Ud}$-irregular point if and only if
\begin{equation*}
	\inf_{K}\balayage^1_{(\Ec \setminus U) \cap K}(x_{\rm o}) =0,
\end{equation*}
where the infimum is taken on the family of compact neighborhoods $K$ of $x_{\rm o}$ ordered by inclusion.
\end{thm}

\section{The {Perron-Weiner-Brelot-Bauer} solution for~$\Lc$}\label{sec:4}
We consider the Dirichlet problem 
\begin{equation} \label{s4 0}
\begin{cases}
		\Lc u = 0 & \quad  \text{in}\ U,\\
		u = \p & \quad  \text{in}\ \partial U		
\end{cases}
\end{equation}
where $U$ is an open subset of~$\r^{N+1}$, $\p \in {C}_c(\partial U)$ and $\Lc$ is the operator defined in \eqref{e-Kolm} satisfying hypothesis {\bf[H.1]}. As we are interested in classical solution to $\Lc u = 0$, throughout the sequel of this article we denote with $\Hd$ the harmonic sheaf defined as
\begin{equation} \label{sheaf of solution}
	U \longmapsto \Hd(U):=\big\{u \in {C}^\infty(U): \Lc u = 0 \mbox{ in } U \big\},
\end{equation}
for every open set $U$ of~$\r^{N+1}$ and we say that a function $u$ is harmonic in an open set $U$ if $u \in \Hd(U)$. 

We next discuss the main steps of the procedure that provides us with the unique solution $u$ to the boundary value problem \eqref{s4 0}.

\subsection{Definition of the sweeping system}

We construct the {{Perron-Weiner-Brelot-Bauer}} solution to problem \eqref{s4 0}. With this aim we consider, for any $z_{\rm o} = (x_{\rm o},t_{\rm o}) \in \r^{N+1}, T > t_{\rm o}$ and $r>0$, the cylinder ${Q}_{r,T}(z_{\rm o})$ defined in \eqref{Kolm cylinder}, and the relevant Dirichlet problem \eqref{pbm on cylinder}. Note that, in the simplest case of the heat operator $\Lc = \Delta -\partial_t$, we are considering the usual Cauchy-Dirichlet problem on the parabolic cylinder
\begin{equation} \label{parabolic cylinder}
	{Q}_r(z_{\rm o}):= B_r(x_{\rm o}) \times ( t_{\rm o}, T). 
\end{equation}

As already recalled in Section \ref{sec:2}, there exists a unique classical solution $u \in {C}^\infty ({Q}_{r,T}(z_{\rm o}))$ to the Dirichlet problem \eqref{pbm on cylinder}, which attains the boundary data on $\partial_P {Q}_r(z_{\rm o})$. By Riesz's Theorem we have that there exists a Radon measure $\mu^{{Q}_{r,T}(z_{\rm o})}_z$, supported on $\partial_P{Q}_{r,T}(z_{\rm o})$, such that 
\begin{equation*} 
    u(z):= \int_{\partial_P{Q}_{r,T}(z_{\rm o})}\p(\zeta){\rm d} \mu^{{Q}_{r,T}(z_{\rm o})}_z(\zeta), \quad \forall z \in {Q}_{r,T}(z_{\rm o}).
\end{equation*}
Thus, the family
\begin{equation} \label{sweep}
	\Omega := \biggl\{ \mu^{{Q}_{r,T}(z_{\rm o})}:=\{\mu^{{Q}_{r,T}(z_{\rm o})}_z\}_{z \in {Q}_{r,T}(z_{\rm o})}:	z_{\rm o} \in \r^{N+1},r\in \r^+, T >t_{\rm o}\biggl\},
\end{equation}
is a sweeping system on $\r^{N+1}$. 

\subsection{The hyperharmonic sheaf $\Ud$ and the $\mathfrak{P}$-harmonic space $(\r^{N+1},\Ud)$}

We now consider the hyperharmonic sheaf $\Ud$ generated by $\Omega$ in accordance with the Definition \ref{hyper sheaf generated by Omega}, and we prove that $(\r^{N+1},\Ud)$ is a $\mathfrak{P}$-harmonic space, according to Definition \ref{beta armonico}. 

We first prove that $(\r^{N+1},\Ud)$ is a harmonic space in the sense of Definition \ref{harmonic space}. We postpone the proof of  axiom {\rm (A2)}, since it is the most involved. The axiom {\rm (A1)} holds because the constant functions are $\Hd_{\Ud}$-functions. The validity of the axiom {\rm (A3)} is a direct consequence of the fact that 
$$
	\mathdutchcal{Q}:= \biggl\{{Q}_{r,T}(z_{\rm o}): z_{\rm o} = (x_{\rm o},t_{\rm o}) \in \r^{N+1}, T> t_{\rm o}, r >0 \biggr\},
$$
is a basis of resolutive sets for the Euclidean topology on $\r^{N+1}$. The axiom {\rm (A4)} follows from the fact that $\Ud$ is the hyperharmonic sheaf generated by $\Omega$. 

Let us focus our attention to the proof of  axiom {\rm (A2)}. We show that $(\r^{N+1},\Ud)$ has the Doob convergence property (axiom {\rm (A2)'}). With this aim, we consider a monotone increasing sequence of $\Hd_{\Ud}$-functions $\{u_n\}_{n \in \mathbb{N}}$ in an open set $U \subset \r^{N+1}$ such that the set
$$
V :=\biggl\{z \in U | \sup_{n \in \mathbb{N}}u_n(z) <\infty \biggr\},
$$
is dense in $U$. We plan to prove that 
\begin{equation}
	\label{s4 2}
u:= \lim_{n \rightarrow +\infty} u_n,
\end{equation}
is a $\Hd_{\Ud}$-function in $U$. We first prove that $\{u_n\}_{n \in \mathbb{N}}$ converges uniformly on every compact subsets $K$ of $U$. For every $z \in K$, we choose a point $(\xi, \tau) \in \R^{N+1}$, and two positive constants $T$ and $r$ such that
$$
{Q}_{r,T}(\xi, \tau) \Subset U \quad \mbox{and} \quad z \in Q^-.
$$
Note that, for every $p \in \mathbb{N}$, $\{u_{n+p}-u_n\}_{n \in \mathbb{N}}$ is a sequence of non-negative $\Hd_{\Ud}$-functions. Then, by the Harnack inequality stated in Theorem \ref{harnack}, we obtain
\begin{align*}
	0 \leq (u_{n+p}(z)-u_n(z)) \leq & \max_{\overline{Q^-}} (u_{n+p}-u_n)\\
	\leq & c \min_{\overline{Q^+}}(u_{n+p}-u_n)\\
	\leq & c (u_{n+p}(\zeta)-u_n(\zeta)) \xrightarrow{n \rightarrow \infty}0.
\end{align*}
In the last inequality we have used the fact that there exists a point $\zeta \in Q^+ \cap V$ such that 
$$
\min_{\overline{Q^+}}(u_{n+p}-u_n)
\leq  c (u_{n+p}(\zeta)-u_n(\zeta)),
$$ 
since $V$ is dense in $U$. From the compactness of $K$ it follows that there exists a finite family of cylinders $\{Q_i\}_{i=1}^{\bar{m}}$,  contained in $U$, such that $K \subset \bigcup_{i=1}^{\bar{m}} Q^-_i$. This proves that $\{u_n\}_{n \in \mathbb{N}}$ converges uniformly on $K$. 
 
We next show that $u$ in \eqref{s4 2} is a $\Hd_{\Ud}$-function. Indeed, fixed ${Q}_{r,T}(z_{\rm o}) \Subset U$, we have that 
\begin{equation*}
	u_n(z) = \int_{\partial_P {Q}_{r,T}(z_{\rm o})} u_n(\zeta) {\rm d}\mu^{{Q}_{r,T}(z_{\rm o})}_z(\zeta), \qquad \forall z \in {Q}_{r,T}(z_{\rm o}), \, \forall n \in \mathbb{N}.
\end{equation*}
From the uniform convergence, it follows that the limit function $u$ satisfies the same identity. Then, $u \in \Hd_{\Ud}({Q}_{r,T}(z_{\rm o}))$ and the Doob convergence property follows. This completes the proof of {\rm (A2)}, and thus, that $(\r^{N+1},\Ud)$ is a harmonic space. 

In order to show that $(\r^{N+1},\Ud)$ is a $\mathfrak{P}$-harmonic space we show that $\Ud$ \textit{separates the points}, so that we can rely on Proposition \ref {thm 1}. Let us consider two points $z_1 \neq z_2 $. There exists $T>0$ such that $z_1,z_2 \in U= \r^N \times [-T,T]$. If $t_1 \neq t_2$, then we set $u_1(z)= e^t$. If otherwise $t_1 = t_2=\tilde{t}$ we choose $\gamma \in \r^N$ so that $\langle x_1 -x_2, e^{\tilde{t}B}\gamma \rangle \neq 0$ and $c>0$ so that
$$
u_2(z) = c - \langle x, e^{t B}\gamma \rangle >0,\qquad \forall z \in U.
$$
By definition $u_2$ satisfies
$$
\Lc u_2 (z) = \langle Bx, \nabla u_2(z)\rangle -\partial_t u_2(z)= - \langle Bx, e^{t B}\gamma \rangle + \langle Bx, e^{t B}\gamma \rangle =0.
$$
Hence, $u_{1,2}(z_1) \neq u_{1,2}(z_2)$ and $u_1,u_2$ are both non-negative superharmonic functions. Then, by Proposition \ref{thm 1} it follows that $(\r^{N+1},\Ud)$ is a $\mathfrak{P}$-harmonic space.

\subsection{Conclusions}

Thanks to Theorem \ref{thm 2}, we conclude that there exists a generalized solution $H^U_\p \in \Hd_{\Ud}(U)$ to the problem 
\begin{equation} \label{pbm 2}
\begin{cases}
		u \in \Hd_{\Ud}(U),\\
		u=\p, \quad  \text{in} \  \partial U, \ \forall \p \in {C}_c(\partial U).
\end{cases}
\end{equation}
We next show that the generalized solution $H^U_\p$ to \eqref{pbm 2} is also a classical solution to the equation $\Lc H^U_\p =0$, then it is a solution to problem \eqref{s4 0}. This fact is the main consequence of the following

\begin{prop} \label{sec4 prop1}
The sweeping system $\Omega$, defined above, is a $\Hd$-sweeping system, with respect to the sheaf $\Hd$ defined in \textup{\eqref{sheaf of solution}}.
 \end{prop}
\begin{proof}
We show that $\mu^{{Q}_{r,T}(z_{\rm o})}$ is a $\Hd$-sweeping, according to Definition \ref{hsweeping}. Clearly fixed $\p \in {C}_c(\partial_P {Q}_{r,T}(z_{\rm o}))$ the function {(recall that $\mu^{{Q}_{r,T}}_z$ is supported on $\partial_P {{Q}_{r,T}}$ )}
$$
\begin{aligned}
	\mu^{{Q}_{r,T}(z_{\rm o})}_\p  \colon {Q}_{r,T}(z_{\rm o})&\to (-\infty,+\infty],\\
	z & \longmapsto \mu^{{Q}_{r,T}(z_{\rm o})}_\p (z) := \int_{\partial {Q}_{r,T}(z_{\rm o})} \p(\zeta) {\rm d}\mu^{{Q}_{r,T}(z_{\rm o})}_z(\zeta),
\end{aligned}
$$
is a $\Hd$-function, because it is the solution $H^{{Q}_{r,T}(z_{\rm o})}_\p$ of the Dirichlet problem \eqref{pbm on cylinder} with boundary data $\p$. Consider $ u \in \Hd(U)$, and let ${Q}_{r,T}(z_{\rm o}) \Subset U$ be any cylinder. Since $u$ is the solution to \eqref{pbm on cylinder} with boundary data $\p=u$, we have that
$$
\mu^{{Q}_{r,T}(z_{\rm o})}_u(z):=\int_{\partial_P {Q}_{r,T}(z_{\rm o})} u(\zeta) {\rm d}\mu^{{Q}_{r,T}(z_{\rm o})}_z(\zeta) = u(z), \qquad \forall z \in {Q}_{r,T}(z_{\rm o}).
$$
Hence, $\mu^{{Q}_{r,T}(z_{\rm o})}_u = u$ on ${Q}_{r,T}(z_{\rm o})$ and the thesis follows.
\end{proof}

The main consequence of the above Proposition is that  $U \subseteq \r^{N+1}$ $\Hd_{\Ud}(U) \equiv \Hd(U)$ for every open set. We are now ready to give the 

\begin{proof}[\bf Proof of Theorem \ref{main result}] Consider the sweeping system $\Omega$ defined in \eqref{sweep} and the hyperharmonic sheaf $\Ud$ generated by $\Omega$. We have that $(\r^{N+1},\Ud)$ is a $\mathfrak{P}$-harmonic space, according to Definition \ref{beta armonico}, then
Theorem \ref{thm 2} provides us with the existence of the {{Perron-Weiner-Brelot-Bauer}} solution $H^U_\p$ to \eqref{pbm 2}. Moreover, Proposition \ref{sec4 prop1} implies that $H^U_\p$ to \eqref{pbm 2} is also a classical solution to the equation $\Lc H^U_\p =0$, then it is a solution to problem \eqref{s4 0}. 
\end{proof}

\section{The Wiener-type test and the cone condition at~$\{t=0\}$}\label{sec:5}
In Section \ref{sec:4} we have shown that there exists the generalized solution $H^U_\p$ to the Dirichlet problem for the operator $\Lc$ defined in \eqref{e-Kolm} in an arbitrary open set $U$
$$
\begin{cases}
\Lc u = 0, & \quad \mbox{in } U,\\
u=\p, & \quad \text{in}\  \partial U, \ \forall \p \in {C}_c(\partial U).
\end{cases}
$$
In this {section} we describe the conditions under which the boundary datum $\p$ is attained by the generalized solution $H^U_\p$. In particular, we prove Theorem \ref{main result2} and  Proposition \ref{main result3}. 

\subsection{Boundary regularity, $\Lc$-potential and $\Lc$-capacity}
In order to use the abstract Theorem \ref{sec2 lemma} we begin showing that every singleton $\{z_{\rm o}\}$ is a polar set in $\r^{N+1}$. Our proof follows the same line as \cite[Lemma 4.5]{LU10}.

Let us consider the fundamental solution $\Gamma$ of the operator $\Lc$, defined in \eqref{fund-sol}. In order to make $\Gamma$ a lower semi-continuous function on $\r^{N+1}$ we agree to let $\Gamma(\zeta;\zeta)=0$, so that
\begin{equation}\label{fund-sol l.s.c.}
\Gamma(\zeta;\zeta) = \liminf_{z \rightarrow \zeta} \Gamma (z;\zeta), \qquad
 \forall \zeta \in \r^{N+1}.
\end{equation}
The following lemma holds.
\begin{lemma}\label{s5 polar lem1}
Let $\zeta_{\rm o} := (\xi_{\rm o},\tau_{\rm o}) \in \r^{N+1}$ be fixed and let $u$ be the function defined as follows
\begin{equation}\label{s5 polar lem2}
u(z):= \Gamma(z;\zeta_{\rm o}) \quad z \in \r^{N+1}.
\end{equation}
Then, $u$ is a non-negative $\Ud$-function on $\r^{N+1}$.
\end{lemma}
\begin{proof}
The non-negativity  and the lower semi-continuity of $u$ follow form  the properties of the fundamental solution $\Gamma$ and from \eqref{fund-sol l.s.c.}. Let us prove that, fixed a cylinder ${Q}_{r,T} \equiv {Q}_{r,T}(z_{\rm o})$, $u$ satisfies the inequality
\begin{equation*}
u(z)  \geq \int_{\partial{Q}_{r,T}} u(\zeta) \, {\rm d}\mu^{{Q}_{r,T}}_z(\zeta) \quad \forall z \in {Q}_{r,T}.
\end{equation*}
Let us consider a function $\p \in {C}(\partial {Q}_{r,T})$ such that $\p \leq u$ on {$\partial {Q}_{r,T}$}.

Assume that $\zeta_{\rm o} \not \in {Q}_{r,T}$ and indicate with $H^{{Q}_{r,T}}_\p$ the generalized solution to the Dirichlet problem in ${Q}_{r,T}$ with boundary datum $\p$. { For any $\delta>0$, let us apply the strong maximum principle in {\cite[Proposition 3.1]{LU10}} on $Q_{r,T} \cap \{t \leq T -\delta\}$. In particular, let us note that $\partial {Q}_{r,T} \cap \{t \leq T-\delta\} = \partial_P {Q}_{r,T} \cap \{t \leq T-\delta\}$. 
Then, since every boundary points of $\partial_P Q_{r,T}$ is $\Lc$-regular (see for instance Proposition A.1 in \cite{Mon96}),} form the harmonicity of $u$ in ${Q}_{r,T}$ and the lower-semicontinuity of $u$, we get that
\begin{equation}\label{eq:maximum}
\liminf_{z \rightarrow \zeta}(u(z)-H^{{Q}_{r,T}}_\p(z)) \geq u(\zeta) -\p(\zeta) \geq 0 \quad {\forall \zeta \in \partial {Q}_{r,T} \cap \{t \leq T -\delta\}}.
\end{equation}
{Hence, \cite[Proposition 3.10]{LU10} we have that $u \geq H^{{Q}_{r,T}}_\p$ on ${Q}_{r,T} \cap \{t \leq T -\delta\}$
\begin{equation}\label{eq:super}
u(z) \geq H^{{Q}_{r,T}}_\p(z) := \int_{\partial{Q}_{r,T}} \p \, {\rm d}\mu^{{Q}_{r,T}}_z \quad \forall z \in  {Q}_{r,T} \cap \{t \leq T -\delta\}.
\end{equation}
Since for any interior point $z= (x,t) \in {Q}_{r,T}$ we can find $\delta>0$ such that $t_{\rm o} < t \leq T -\delta < T$, we have that \eqref{eq:super} holds true for any $z \in {Q}_{r,T}$. Then,} passing to the supremum of every $\p \leq u$ on $ \partial {Q}_{r,T}$ we obtain the desired inequality \eqref{s5 polar lem2}.

On the other hand, let us suppose that $\zeta_{\rm o} \in {Q}_{r,T}$. Since $u \equiv 0$ on the set $\{t \leq \tau_{\rm o}\}$ we have that $\p \leq 0 $ on $\partial {Q}_{r,T} \cap \{t \leq \tau_{\rm o}\}$, { so that by  \cite[Proposition 3.10]{LU10} 
	we obtain $H^{{Q}_{r,T}}_\p \leq 0$ on ${Q}_{r,T} \cap \{t \leq \tau_{\rm o}\}$. In particular,}
$H^{{Q}_{r,T}}_\p(\zeta_{\rm o}) \leq 0$. Let us consider $\widetilde{Q}_{r,T}:= {Q}_{r,T} \setminus \{\zeta_{\rm o}\}$. 
{Then, by proceeding as in the previous case we can prove \eqref{eq:maximum} in $\widetilde{Q}_{r,T}$, which yields that }
$u \geq H^{{Q}_{r,T}}_\p$ on {$\widetilde{Q}_{r,T} \cap \{t \leq T-\delta\}$}. Moreover in $\zeta_{\rm o}$ we have that $u(\zeta_{\rm o}) =0 \geq H^{{Q}_{r,T}}_\p(\zeta_{\rm o})$. Hence, $u \geq H^{{Q}_{r,T}}_\p$ on {${Q}_{r,T} \cap \{t \leq T-\delta\}$} and {\eqref{eq:super}, and in turn} \eqref{s5 polar lem2}, follows exactly as for the case $\zeta_{\rm o} \not \in {Q}_{r,T}$.
\end{proof}

\begin{prop}\label{s5 polar prop}
Every singleton $\{z_{\rm o}\}$, $z_{\rm o} \in \r^{N+1}$, is a polar set in $\r^{N+1}$.
\end{prop}
\begin{proof}
Let $z_{\rm o} :=(x_{\rm o},t_{\rm o}) \in \r^{N+1}$ and we use the fundamental solution $\Gamma$ to built a function $p$ which satisfies the condition of Definition \ref{polar}.

For any $\eps>0$ let us consider the family of points 
 $$
 \zeta_\eps:=(\xi_\eps,\tau_\eps)=(e^{\eps B}x_{\rm o},t_{\rm o}-\eps).
 $$ 
 By the definition \eqref{fund-sol} of fundamental solution we obtain that
$$
\Gamma(z_{\rm o};\zeta_\eps) := \frac{(4\pi)^{N/2}}{\sqrt{\textup{det}C(t_{\rm o}-\eps,t_{\rm o})}} \xrightarrow{\eps \rightarrow 0^+} +\infty.
$$
Then, there exists a decreasing sequence $\{\eps_n\}_{n \in \mathbb{N}}$ such that
\begin{equation}\label{s5 polar2}
\Gamma(z_{\rm o};\zeta_{\eps_n}) \geq 4^n, \qquad \forall n \in \mathbb{N}.
\end{equation}
Let us consider the function $p$ defined as follows
\begin{equation*}
p(z) := \sum_{n =1}^\infty \frac{\Gamma(z;\zeta_{\eps_n})}{2^n}
\end{equation*}
and show that $p$ satisfies the condition of Definition \ref{polar}. By assumption \eqref{s5 polar2} we have
$$
p(z_{\rm o})= \sum_{n =1}^\infty \frac{\Gamma(z_{\rm o};\zeta_{\eps_n})}{2^n}\geq \sum_{n =1}^\infty 2^n = +\infty.
$$
Let $z \neq z_{\rm o}$. Then, there exists a positive $r$ such that
$$
\overline{B}_r(z) \times \overline{B}_r(z_{\rm o}) \cap \{(w,\zeta) \in \r^{N+1} \times \r^{N+1} : w = \zeta\} = \emptyset.
$$
Since $\Gamma$ is continuous in $\{(w,\zeta) \in \r^{N+1} \times \r^{N+1} : w \neq \zeta\}$, we have 
\begin{equation*}
M:= \max_{(y,s) \in \overline{B}_r(z) \atop (\eta, \sigma) \in \overline{B}_r(z_{\rm o})}\Gamma(y,s;\eta,\sigma) < +\infty.
\end{equation*}
Moreover, since $\zeta_{\eps_n} \rightarrow z_{\rm o}$ as $n \rightarrow +\infty$, there exists an index $\bar{n}>0$ such that, for any $n > \bar{n}$, $\zeta_{\eps_n} \in \overline{B}_r(z_{\rm o})$.
We show that $p$ converges uniformly on $\overline{B}_r(z)$. Indeed, 
$$
\sum_{n =\bar{n}}^{\infty} \,\frac{ \sup_{w \in \overline{B}_r(z)}|\Gamma(w;\zeta_{\eps_n})|}{2^n} \leq M \sum_{n= 1}^{\infty}\frac{1}{2^n} =M.
$$
Hence, it follows that $p$ converges uniformly on $\overline{B}_r(z)$. From Lemma \ref{s5 polar lem1} it then follows that $p$ is a $\Ud$- function on $\r^{N+1}$, finite for any $z \neq z_{\rm o}$. Moreover, since the harmonic space $(\r^{N+1},\Ud)$ has the Doob convergence property, by Remark \ref{sup finite on a dense set}, $p$ is superharmonic. Hence, $\{z_{\rm o}\}$ is a polar set in $\r^{N+1}$, according to Definition \ref{polar}, with associated function $p$.
\end{proof}
Now we apply the results presented in the last part of Chapter 3 to discuss the regularity of boundary points. Since the sweeping system $\Omega$, defined in \eqref{sweep}, is a $\Hd$-sweeping (see Proposition \textup{\ref{sec4 prop1}}), the definition of $\Lc$-regular point and $\Hd_{\Ud}$-regular point coincide. Indicating with $U$ an open subset of~$\r^{N+1}$ with non-empty boundary
 let us consider, for any $r>0$ and $z_{\rm o} \in \partial U$, the set $\mathcal{Q}_r(x_{\rm o},0)$ defined in \eqref{eq:cylinder-homogeneous}
 and  let
 \begin{equation}\label{ball}
     G_r:=  \{ (x,t) \in \r^{N+1}\setminus U: -r^2 < t \leq 0, \ \| x-e^{-tB}x_{\rm o}\| \leq r\}.
 \end{equation}
As consequence of Theorem \ref{sec2 lemma} and Proposition \ref{s5 polar prop} we characterize the regularity of boundary points as follows.

\begin{corol}\label{reg condition}
	Let $U \subseteq \r^{N+1}$ be an open set and let $z_{\rm o} \in \partial U$. Then, being $G_r$ the set defined in \eqref{ball}, we have that $z_{\rm o}$ is a $\Lc$-regular point if and only if
	\begin{equation*}
		\lim_{r \rightarrow 0^+}\balayage^1_{	G_r}(z_{\rm o})>0.
	\end{equation*}
\end{corol}

Before proceeding with the proof of our regularity criteria we still need few more definitions. Let us denote with $\mathcal{M}(\r^{N+1})$ the collection of all nonnegative Radon measure on $\r^{N+1}$ and call 
\[
\Gamma_\mu(z) :=\int_{\r^{N+1}} \Gamma(z,\zeta)\ {\rm d}\mu(\zeta),\qquad z\in \r^{N+1},
\]
the {\it $\Lc$-potential of $\mu$}.

If $F$ is a compact set of $\r^{N+1}$ and $\mathcal{M}(F)$ is the collection of all nonnegative Radon measure on $\r^{N+1}$ with support in $F$,
the {\it $\Lc$-capacity of $F$} is defined as
\[
\mathrm{cap} (F) := \sup\{ \mu(F)\ | \ \mu\in\mathcal{M}(F), \ \Gamma_\mu \le 1 \mbox{ on } \r^{N+1}\}.
\]
We list some properties of the $\Lc$-capacities. For every $F$, $F_1$ and $F_2$ compact subsets of $\r^{N+1}$, we have:
\begin{enumerate}[\it (i)]
\item $\mathrm{cap} (F) < \infty $;
\item if $F_1\subseteq F_2$, then $\mathrm{cap} (F_1) \le \mathrm{cap} (F_2)$;
\item $\mathrm{cap} (F_1 \cup F_2) \le \mathrm{cap}(F_1) + \mathrm{cap} (F_2)$;
\item $\mathrm{cap} (\delta_r (F)) = r^{Q} \mathrm{cap} ( F)$ for every $r>0$.
\end{enumerate} 
The properties $(i)-(iv)$ are quite standard, and they follow from the features of $\Gamma$. Following the same lines of the proof of \cite[Teorema 1.1]{L73}, we have the existence of a unique measure $\mu_F\in \mathcal{M}(F)$ such that 
\[
\balayage_F^1(z)=\Gamma_{\mu_F}(z)=\int_{\r^{N+1}}{\Gamma(z,\zeta)\,\rm{d}\mu_F(\zeta)}\quad \forall \, z\in\r^{N+1},
\]
and 
\[
\mu_F(\r^{N+1})=\mathrm{cap} (F).
\]
The proof of this fact relies on the good behavior of $\Gamma$, a representation formula of Riesz-type 
for superharmonic functions proved in \cite[Theorem 5.1]{CL09}, and a Maximum Principle for $\Lc$. 

\subsection{Wiener-type criterium}
We begin proving Theorem~\ref{main result2}. We extend to our contest the same approach used in \cite{KLT18}. We will use the lemma below; see \cite[Lemma 5.1]{KLT18}.
\begin{lemma}
	\label{sec6 lem1}
For every $n \in \mathbb{N}$, let us split the set $G_r$ in \eqref{ball} as follows
$$
\	G_r = 	G_r^n \cup 	G_r^{*n},
$$
where, for any $\lambda \in (0,1)$, we write
$$
G_r^n = \{z \in G_r : \Gamma(z_{\rm o};z) \geq \lambda^{-n\log n}\} \cup \{z_{\rm o}\}
$$
$$
and \qquad G_r^{*n}=\{z \in G_r : \Gamma(z_{\rm o};z) \leq \lambda^{-n\log n}\}.
$$
Then,
$$
\lim_{r \rightarrow 0^+}\balayage^1_{G_r}(z_{\rm o})= \lim_{r \rightarrow 0^+}\balayage^1_{G_r^n}(z_{\rm o}).
$$
\end{lemma}

\begin{proof}[\bf Proof of the necessary condition in Theorem~{\rm \ref{main result2}}] 
We prove the implication
\begin{equation}\label{necessary cond}
z_{\rm o} \mbox{ is $\Lc$-regular} \Rightarrow \sum_{n =1}^{\infty}\balayage^1_{U^c_n(z_{\rm o})}(z_{\rm o}) = +\infty,
\end{equation}
By the hypothesis it follows from Corollary \ref{reg condition} that
\begin{equation}\label{absurd}
\lim_{r \rightarrow 0^+}\balayage^1_{G_r}(z_{\rm o})>0.
\end{equation}
Let us assume by contradiction that 
\begin{equation}\label{contr}
 \sum_{n =1}^{\infty}\balayage^1_{U^c_n(z_{\rm o})}(z_{\rm o}) <+\infty,
\end{equation}
where~$U^c_n(z_{\rm o})$ is the set defined in~\eqref{level set}. 
We are going to prove that the assumption~\eqref{contr} is in contradiction with~\eqref{absurd}. 

By hypothesis~\eqref{contr}, for every~$\eps>0$, there exists~$n_\eps:= n(\eps)\in \mathbb{N}$ such that
$$
 \sum_{n=n_\eps}^{\infty}\balayage^1_{U^c_n(z_{\rm o})}(z_{\rm o}) <\eps.
$$
On the other hand, following the notation of Lemma~\ref{sec6 lem1}, for any positive radius~$r>0$, we have
$$
G_r^{n_\eps} \subseteq \bigcup_{n = n_\eps}^\infty U^c_n(z_{\rm o}).
$$
Then, by Proposition~\ref{subadd of balayage}, we get
$$
\balayage^1_{G_r^{n_\eps}}(z_{\rm o}) \leq  \sum  _{n =n_\eps}^{\infty}\balayage^1_{U^c_n(z_{\rm o})}(z_{\rm o}) < \eps.
$$
By Lemma \ref{sec6 lem1}, we get
$$
\lim_{r \rightarrow 0^+}\balayage^1_{G_r}(z_{\rm o})=0,
$$
which is in contradiction with \eqref{absurd}. This prove the necessary condition~\eqref{necessary cond}.
\end{proof}

We now prove the sufficient condition of Theorem~\ref{main result2}. This will require three lemmas. The first follows by a similar path as in \cite[Lemma 6.1]{KLT18} by relying on Corollary~\ref{open set is mp-set}

\begin{lemma}
Suppose we have a sequence of compact sets $\{F_n\}_{n\in \mathbb{N}}$ in $\r^{N+1}$ such that
$$
\begin{cases}
F_n\cap F_k=\emptyset & \text{ if }n\neq k,\\
\forall r>0\,\,\, \exists \,\bar{n} \,\,\mbox{ such that }\,\, F_n\subseteq G_r &  \text{for }n\geq\bar{n}.
\end{cases}
$$
Suppose also that the following two conditions hold true:
\begin{itemize}

\item[($i$)]$$\sum_{n=1}^{+\infty}{\balayage^1_{F_k}(z_0)}=+\infty;$$
\item[($ii$)]$$\sup_{n\neq k}\sup{\left\{\frac{\Gamma(z,\zeta)}{\Gamma(z_0,\zeta)}\,:\,z\in F_n,\,\zeta\in F_k\right\}}\leq M_0.$$
\end{itemize}
Then we have~$\balayage^1_{G_r}(z_0)\geq \frac{1}{2M_0}$ for every positive~$r$.
\end{lemma}

 Now, for any fixed $\lambda\in(0,1)$, we recall that the definition of the set $U_n^c(z_{\rm o}) \equiv U_n^c(x_{\rm o},0)$ given in~\eqref{level set} and, setting~$\alpha(n)=n\log{n}$, let us  denote
\[
T_n :=\max_{(x,t)\in U_n^c(x_{\rm o},0)}  - t = \big(c_N \lambda^{\alpha(n)} \big)^\frac{2}{Q}\,,
\]
where~$Q$ is the homogeneous dimension associated to $(D_r)_{r>0}$ and $c_N$ is a given dimensional constant. Fix~$q\in\mathbb{N}$ such that 
  \begin{equation}\label{def:q}
  q\geq q_{\rm o}:=4+\frac{m}{\log{\left(\frac{1}{\lambda}\right)}},\quad\mbox{where }m=
  \max{\left\{2,\frac{2Q}{\log{6}},\frac{2\sigma_C^2}{\log{6}},\frac{Q\log{2}}{\log{8}},\frac{2Q\log{(\frac{\kappa+1}{\sigma})}}{\log{8}}\right\}},
  \end{equation}
 and $\sigma_C, \sigma$ are the constants in \eqref{eq:propC} and \eqref{confrhomnonhom}. Moreover, let us also denote by
 \begin{equation}\label{eq:choice-p}
 p=1+\Big\lfloor\frac{q}{2}\Big\rfloor :=1+\mbox{the integer part of }\frac{q}{2}.
 \end{equation}
 So $\frac{q}{2}\leq p\leq 1 + \frac{q}{2} <q-1$. 
 For any $n\in\mathbb{N}$  consider the sets
 \begin{eqnarray*}
 U_{nq}^c(x_{\rm o},0) & = & \big( U_{nq}^c(x_{\rm o},0)\cap \{t \geq -T_{nq+p}\}\big) \cup \ \big(U_{nq}^c(x_{\rm o},0)\cap \{t \leq -T_{nq+p}\}\big)\\
 & : = & F_n^{( {\rm o})}\cup F_n.
 \end{eqnarray*} 
 Moreover, notice that, since $nq+p<q(n+1)$, then
  \begin{equation}\label{eq:ordine}
  \min_{(x,t)\in F_{m}}{t}=-T_{mq}>-T_{nq+p}=\max_{(\xi,\tau)\in F_{n}}{\tau}\qquad\forall n,m\in\mathbb{N},\,\,m>n.
  \end{equation}
 
Let us prove the following lemma.
\begin{lemma}
With the notation above, we have that
\[
\sum_{n=1}^{+\infty}{\balayage^1_{F_n^{( {\rm o})}}(z_{\rm o})}<+\infty.
\]
\end{lemma}
\begin{proof}
We are going to prove that $F_n^{( {\rm o})}$ is contained in a homogeneous cylinder $\mathcal{Q}_{r_n}$ so that
\begin{equation}\label{claimi}
\sum_{n=1}^{+\infty}{\left(\frac{1}{\lambda}\right)^{\alpha(nq+1)}r_n^Q}<+\infty.
\end{equation}
This is enough to prove the statement since
\[
\balayage^1_{F_n^{( {\rm o})}}(z_{\rm o})=\int_{F_n^{( {\rm o})}}{\Gamma(z_{\rm o};\zeta)\,\rm{d}\mu_{F_n^{( {\rm o})}}(\zeta)}\leq\left(\frac{1}{\lambda}\right)^{\alpha(nq+1)}\textup{cap} (F_n^{( {\rm o})}),
\]
and by monotonicity and homogeneity we have 
\[
\textup{cap} (F_n^{( {\rm o})})\leq \textup{cap} (\mathcal{Q}_{r_n})= \textup{cap} (\mathcal{Q}_{1}) r_n^Q.
\]
In order to prove \eqref{claimi}, we have to find a good bound for $r_n$. Fix $z=(x,t)\in F_n^{( {\rm o})}$. Since in particular $z\in U_{nq}^c(x_{\rm o},0)$, we have that by definition of $|\cdot|_C$ in \eqref{def:normC}
\begin{eqnarray*}
\left|e^{-B}D_{\frac{1}{\sqrt{-t}}}(x-e^{-t B}x_{\rm o})\right|_C & = &   \frac{1}{4}\langle C^{-1}\left(-1,0\right)e^{-B}D_{\frac{1}{\sqrt{-t}}}(x-e^{-tB}x_{\rm o}),e^{-B}D_{\frac{1}{\sqrt{-t}}}(x-e^{-tB}x_{\rm o})\rangle\notag\\
& \leq & 
\log{\left(\frac{c\lambda^{\alpha(nq)}}{(-t)^{\frac{Q}{2}}}\right)}\,,
\end{eqnarray*}
while, on the other hand, by \eqref{eq:propC}, we get
\[
    \left|e^{-B}D_{\frac{1}{\sqrt{-t}}}(x-e^{-tB}x_{\rm o})\right|_C^2 \geq  
 \sigma^2_C \left|D_{\frac{1}{\sqrt{-t}}}(x-e^{-tB}x_{\rm o})\right|^2\,,
\]
so then
\begin{equation}\label{boundM}
\left|D_{\frac{1}{\sqrt{-t}}}(x-e^{-tB}x_{\rm o})\right|^2\leq\frac{1}{\sigma^2_C }\log{\left(\frac{c_N\lambda^{\alpha(nq)}}{(-t)^{\frac{Q}{2}}}\right)}.
\end{equation} 
Therefore, from \eqref{confrhomnonhom}, we deduce
\begin{eqnarray*}
&& \frac{1}{\sqrt{-t}}\left\|x-e^{-tB}x_{\rm o}\right\|\\*[0.5ex]
&& \quad = \left\|D_{\frac{1}{\sqrt{-t}}}(x-e^{-tB}x_{\rm o})\right\|\\
&& \quad \leq  (\kappa+1)\max{\left\{\left|D_{\frac{1}{\sqrt{-t}}}(x-e^{-tB}x_{\rm o})\right|^\frac{1}{2\vartheta+1}, \left|D_{\frac{1}{\sqrt{-t}}}(x-e^{-tB}x_{\rm o})\right|^{\frac{1}{2(\kappa+\vartheta)+1}}\right\}}\\
&& \quad \leq (\kappa+1)\max{\left\{\frac{1}{\sigma_C^\frac{1}{2\vartheta +1} }\log^{\frac{1}{2(2\vartheta+1)}}{\left(\frac{c\lambda^{\alpha(nq)}}{(-t)^{\frac{Q}{2}}}\right)}, \frac{1}{\sigma_C^\frac{1}{2(\kappa+\vartheta)+1}}\log^{\frac{1}{2(2(\kappa+\vartheta)+1)}}{\left(\frac{c\lambda^{\alpha(nq)}}{(-t)^{\frac{Q}{2}}}\right)}\right\}}.
\end{eqnarray*}
Let us note that from our choice of $\alpha(n)=n\log{n}$ we can check that the sequence $n \mapsto \alpha(nq+p)-\alpha(nq)$ is monotone increasing. In particular, by the choice of $p$ in \eqref{eq:choice-p}, this yields that 
\begin{eqnarray*}
    \alpha(nq+p)-\alpha(nq) & \geq &  \alpha(q+p)-\alpha(q)\\
    & \geq & \alpha\left(\frac{3}{2}q\right)-\alpha(q)  \geq  \frac{1}{2}q\log{\left(\frac{3}{2}q\right)}\geq\frac{1}{2}q\log{6}.
\end{eqnarray*}
By our choice of $q$ \eqref{def:q}, we have that $\alpha(nq+p)-\alpha(nq)\geq\frac{Q}{2\log{(\frac{1}{\lambda})}}$ and so, for any $n \in \mathbb{N}$, it holds
\[
T_{nq+p}^{\frac{Q}{2}} 
 \leq  c\lambda^{\alpha(nq)}e^{-\frac{Q}{2}}.
\]
This and the fact that the functions $s\mapsto s\log^{\beta}\frac{\theta}{s^Q}$ are increasing in the interval $(0,e^{-\beta}\theta^{\frac{1}{Q}}]$ allow to bound the term $\|x-e^{-tB}x_{\rm o}\|$ further. Indeed, having $0<-t\leq T_{nq+p}$, we get
\begin{eqnarray*}
\left\|x-e^{-tB}x_{\rm o}\right\| & \leq &\frac{(\kappa+1)}{\sigma_C} \sqrt{-t}\log^{\frac{1}{2}}{\left(\frac{c\lambda^{\alpha(nq)}}{(-t)^{\frac{Q}{2}}}\right)}\\
& \leq & \frac{(\kappa+1)}{\sigma_C}\sqrt{T_{nq+p}}\log^{\frac{1}{2}}{\left(\frac{c_N\lambda^{\alpha(kq)}}{T_{nq+p}^{\frac{Q}{2}}}\right)}\,,
\end{eqnarray*}
 since given that $\frac{1}{2}q\log{6}\geq \frac{\sigma^2_C}{\log{(\frac{1}{\lambda})}}$ we have also $T_{nq+p}^{\frac{Q}{2}}\leq c\lambda^{\alpha(nq)}e^{-\sigma_C^2}$, which says 
\[
\log^{\frac{1}{2}}{\left(\frac{c\lambda^{\alpha(nq)}}{T_{nq+p}^{\frac{Q}{2}}}\right)}\geq \sigma_C.
\]

Summing up, we have just proved that
\[
(x,t)\in F_k^{( {\rm o})}\qquad\Longrightarrow\quad\begin{cases}
\left\|x-e^{-tB}x_{\rm o}\right\|\leq\frac{(\kappa+1)}{\sigma_C}\sqrt{T_{nq+p}}\log^{\frac{1}{2}}{\left(\frac{c\lambda^{\alpha(nq)}}{T_{nq+p}^{\frac{Q}{2}}}\right)}=:r_n\\
0<-t\leq T_{nq+p}\leq (\kappa+1)^2T_{nq+p} < r_n^2,&  \,
\end{cases}
\]
namely, recalling the definition of $\mathcal{Q}_r(x_{\rm o},0)$ in \eqref{eq:cylinder-homogeneous} yields that
$$
F_k^{( {\rm o})}\subseteq \mathcal{Q}_{r_n}(x_{\rm o},0).
$$
Moreover, by the choice of  $p\geq\frac{q}{2}>1+\frac{1}{\log(\frac{1}{\lambda})}$, we obtain that \eqref{claimi} is verified by the same argument as in \cite[Lemma 6.2]{KLT18}.
\end{proof}

\begin{lemma}
Let $z_{\rm o}:=(x_{\rm o},0)$. There exists a positive constant $M_0$ such that
\[
\frac{\Gamma(z;\zeta)}{\Gamma(z_{\rm o};\zeta)}\leq M_0\quad\forall\,z\in F_m,\,\forall\,\zeta\in F_n,\quad\forall\,m,n\in\mathbb{N},\,\,m\neq n.
\]
\end{lemma}
\begin{proof}
Fix any $m,n\in \mathbb{N}$ with $m\neq n$. If $m\leq n-1$, then $F_m$ lies below $F_n$ implying that $\Gamma(z;\zeta)=0$ by definition of $\Gamma$. Thus, the thesis follows. Hence, with no loss of generality we can assume $m\geq n+1$. For every $z=(x,t)\in F_m$ and $\zeta=(\xi,\tau)\in F_n$, by \eqref{eq:ordine} we have that
\[
\mu=\frac{-t}{-\tau}\leq \frac{-\min\limits_{(x,t) \in F_m}t}{-\max\limits_{(\xi,\tau) \in F_n}\tau} = \frac{T_{mq}}{T_{nq+p}} =\left(\frac{\lambda^{\alpha(mq)}}{\lambda^{\alpha(nq+p)}}\right)^{\frac{2}{Q}}=\left(\frac{1}{\lambda}\right)^{\frac{2}{Q}\big(\alpha(nq+p)-\alpha(mq) \big)}.
\]
Moreover, since $nq+p < q(n+1) \leq mq$, by monotonicity of the map $n \mapsto \alpha(nq+q)-\alpha(nq+p)$ and by the choice of $p$ in \eqref{eq:choice-p}, we have
\begin{eqnarray*}
\alpha(mq)-\alpha(nq+p) & \geq & \alpha(nq+q)-\alpha(nq+p)\\
& \geq & \alpha(2q)-\alpha(q+p)\\
& \geq & \alpha(2q)-\alpha\left(\frac{3}{2}q+1\right) \geq \left(\frac{q}{2}-1\right)\log{(2q)}.
\end{eqnarray*}

Furthermore, by our choice of $q$ \eqref{def:q} we have that
\[
\alpha(mq)-\alpha(nq+p)\geq\left(\frac{q}{2}-1\right)\log{(8)}\geq\frac{Q}{2}\frac{\max{\{\log{2},\log{(\frac{\kappa+1}{\sigma})^2}\}}}{\log{(\frac{1}{\lambda})}}
\]
which, in turn, implies $\mu\leq \min{\{\frac{1}{2},\frac{\sigma^2}{(\kappa+1)^2}\}}$. Hence, by Lemma \ref{rapporto} we get
\begin{eqnarray*}
\frac{\Gamma(z,\zeta)}{\Gamma(z_{\rm o},\zeta)} & \leq & \left(\frac{1}{1-\mu}\right)^{\frac{Q}{2}}e^{C\sqrt{\mu}M(z_{\rm o},z)M(z_{\rm o},\zeta)} \leq   2^{\frac{Q}{2}}\,e^{C\sqrt{\mu}M(z_{\rm o},z)M(z_{\rm o},\zeta)}\,,
\end{eqnarray*}
recalling the notation in \eqref{eq:notation}. 

To finish the proof we need to show that the exponential is uniformly bounded for $z\in F_m$ and $\zeta\in F_n$. By estimating as in \eqref{boundM} we have
\begin{eqnarray*}
\left|D_{\frac{1}{\sqrt{-\tau}}}(\xi-e^{-\tau B}x_{\rm o})\right|^2 & \leq & \frac{1}{\sigma^2_C }\log{\left(\frac{c\lambda^{\alpha(nq)}}{(-\tau)^{\frac{Q}{2}}}\right)}\\
&\leq& \frac{1}{\sigma^2_C }\log{\left(\frac{c\lambda^{\alpha(nq)}}{T_{nq+p}^{\frac{Q}{2}}}\right)}\notag\\
& = & \frac{1}{\sigma^2_C }\log{\left(\frac{1}{\lambda}\right)}(\alpha(nq+p)-\alpha(nq))\,,
\end{eqnarray*}
and in a similar way
\[
\left|D_{\frac{1}{\sqrt{-t}}}(x-e^{-tB}x_{\rm o})\right|^2 \leq  \frac{1}{\sigma^2_C }\log{\left(\frac{1}{\lambda}\right)}(\alpha(mq+p)-\alpha(mq))\,,
\]
so that now the proof follows as in \cite[Lemma 6.3]{KLT18}.
\end{proof}

\begin{proof}[\bf Proof of the sufficient condition in Theorem \ref{main result2}] 
The proof now follows by a similar argument as in \cite[Theorem 1.1]{KLT18} by using the above lemmas.
\end{proof}

\subsection{Proof of Proposition \ref{main result3}}
We prove that $\r^{N+1} \setminus U$ is not thin in~$z_{\rm o}$, according to Definition \ref{thinnes}. Then, by Theorem \ref{thin thm1}, $z_{\rm o}$ is a $\Lc$-regular point. Since $z_{\rm o} \in \r^{N+1} \setminus U$ we show that $(\r^{N+1} \setminus U) \setminus \{z_{\rm o}\}$ is not thin in $\{z_{\rm o}\}$.
Thanks to Proposition \ref{thin prop1} it is enough to prove that, for any open neighborhood $V$ of $z_{\rm o}$
\begin{equation}\label{cone 0}
\reduit^1_{(\r^{N+1} \setminus U)\cap (V \setminus \{z_{\rm o}\})} (z_{\rm o}) = 1,
\end{equation}
For any $r>0$ let us consider the neighborhoods $G_r$ defined in \eqref{ball}. By Proposition~\ref{subadd of balayage} we have that
$$
\reduit^1_{G_r}(z_{\rm o}) \leq \reduit^1_{G_r \setminus \{z_{\rm o}\}}(z_{\rm o})+ \reduit^1_{\{z_{\rm o}\}}(z_{\rm o}).
$$
Form \eqref{cone 0} it follows that it is enough to show that
\begin{equation}\label{cone 1}
\reduit^1_{G_r \setminus \{z_{\rm o}\}}(z_{\rm o}) \geq \reduit^1_{G_r}(z_{\rm o}) - \reduit^1_{\{z_{\rm o}\}}(z_{\rm o}) \geq 1.
\end{equation}
Now, let us adopt the following notation 
\[
\mathcal{C}_r(z_{\rm o}) = (x_{\rm o} + D_rK) \times\{-r^2T\} =: K_r(x_{\rm o}) \times\{-r^2T\}.
\]
For any $\theta>1$ and any $n \in \mathbb{N}$ let us denote with
\[
F_n^{(\theta)}:= \left\{z = (x,t) \in \r^{N+1}: \frac{1}{\lambda^{n\log n}}< \Gamma(x_{\rm o},0;x,t) \leq \frac{\theta}{\lambda^{n\log n}}\right\}.
\]
There exists $\bar{n}\in \mathbb{N}$ such that
\begin{equation}\label{cone99}
F_n^{(\theta)}\cap \mathcal{C}_r(z_{\rm o}) \subset U_n^c(z_{\rm o}) \quad \forall n \geq \bar{n}.
\end{equation}
We claim that there exist $\bar{n}_1\geq\bar{n}$ and a non-empty open set $B\subset \R^{N+1}$ such that 
\begin{equation}\label{eq:claim}
B\subseteq F_n^{(\theta)}\cap\mathcal{C}_r(z_{\rm o}) \qquad \forall n\geq\bar{n}_1.
\end{equation}
Indeed, take $\bar{n}_1\geq\bar{n}$ such that, for  any fixed $ r \in (0,R)$, it holds
\[
\sup_{\xi\in {\rm int}(K_r(x_{\rm o}))}{\Gamma(x_{\rm o},0;\xi,-r^2T)}<\frac{1}{\lambda^{n\log n}} \quad\forall n\geq\bar{n}_1.
\]
Consider 
\[
A:=\left\{\xi \in {\rm int}(K_r(x_{\rm o})),\,\frac{1}{\theta}\Gamma(x_{\rm o},0;\xi,-r^2T)<\frac{1}{\lambda^{n\log n}}<\Gamma(x_{\rm o},0;\xi,-r^2T)\right\},
\]
which is open, and non-empty since ${{\rm int}}(K_r(x_{\rm o}))\neq\emptyset$ and $\theta>1$. Moreover $A \times \{-r^2T\} \subset F^{(\theta)}_n$ by construction, and $A \times \{-r^2T\} \subset \mathcal{C}_r(z_{\rm o})$ for $r \in (0,R)$. Now, note that for sufficiently small $r>0$ we have that
\begin{equation}\label{cone 4}
\mathcal{Q}_r(z_{\rm o}) \setminus U \supseteq   K_{r}(x_{\rm o}) \times \{-r^2 T\}.
\end{equation}

 Thus, by Definition~\ref{balayage} of reduit function it holds~$\reduit^1_{G_r}(z_{\rm o}) = 1 $ and  by
 \[
 \int_{K_r(x_{\rm o}) }\Gamma(x_{\rm o},0; \xi,-r^2T) \, {\rm d}\xi \leq 1,
 \]
 keeping in mind \eqref{cone99} and \eqref{eq:claim}, we obtain that for $n \geq \bar{n}_1$
\begin{eqnarray*}
\reduit^1_{G_r \setminus \{z_{\rm o}\}}(z_{\rm o}) & \geq &  \reduit^1_{G_r}(z_{\rm o})- \reduit^1_{\{z_{\rm o}\}}(z_{\rm o})\\
& \geq & \int_{K_r(x_{\rm o}) }\Gamma(x_{\rm o},0; \xi,-r^2T) \, {\rm d}\xi  - \reduit^1_{\{z_{\rm o}\}}(z_{\rm o}) \geq \frac{|A|}{\lambda^{{n}\log {n}}}  - \reduit^1_{\{z_{\rm o}\}}(z_{\rm o}) \geq 1\,,
\end{eqnarray*}
up to choosing a sufficiently small $\lambda$ such that $\lambda^{n\log n} \leq (1+ \reduit^1_{\{z_{\rm o}\}})/|A|$. 
 Then, condition~\eqref{cone 1} is satisfied and the thesis follows.
\hfill $\square$


\vspace{3mm}



\begin{thebibliography}{99}

                
\bibitem{AP20} {\sc F. Anceschi, S. Polidoro}:
               A survey on the classical theory for Kolmogorov equation.
               {\it Le Matematiche} Vol LXXV- Issue 1~(2020)~221--258.
               \vs
               
  \bibitem{APR23} {\sc F. Anceschi, M. Piccinini, A. Rebucci}:
               New perspectives on recent trends for Kolmogorov operators. To Appear in
               {\it Springer INdAM Series}~(2023).
               \vs
     
\bibitem{BLU07}    {\sc A. Bonfiglioli, E. Lanconelli, F. Uguzzoni}:
                  Stratified Lie Groups and their sub-Laplacians.
                  {\it Springer Monographs in Mathematics}, 2007.    
                  \vs

\bibitem{Bau66} {\sc H. Bauer}:
                 Harmonische R\"aume und ihre Potentialtheorie.
                 {\it Lecture Notes in Mathematics}, 22, Springer-Verlag, 1966. 
                 \vs
                 
\bibitem{Bon69}  {\sc J.~\!-M. Bony}:
                Principe du maximum, inégalité de Harnack et unicité du problème de Cauchy pour les opérateurs elliptiques dégénérés.
                {\it Ann. Inst. Fourier} (Grenoble)~19~(1969)~277--304.
                \vs

                  
\bibitem{Bre60} {\sc M. Brelot}:
                Lectures on Potential Theory.
                {\it Tata Institute of Fundamental Research}, Bombay, 1960.
                 \vs
   
\bibitem{CL09} {\sc C. Cinti, E. Lanconelli}:
                Riesz and Poisson-Jensen representation formulas for a class of ultraparabolic
               	operators on Lie groups.
                {\it Potential Anal.} 30~(2009)~179--200.
                \vs
   
\bibitem{CC72} {\sc C. Constantinescu, C. Cornea}:
               Potential theory on harmonic spaces
               {\it Springer-Verlag}, Berlino, 1972.  
                \vs
             
\bibitem{EK70} {\sc E.~\!G. Effros, J.~\!L. Kazdan}:
               On the Dirichlet problem for the heat equation
               {\it Indiana Univ. Math. J.}~20~(1970/1971)~683--693.
               \vs
   
\bibitem{EG82} {\sc L.~\!C. Evans, R.~\!F. Gariepy}: 
               Wiener’s criterion for the heat equation.
               {\it Arch. Rational Mech. Anal.}~78~(1982)~293–314.
               \vs

\bibitem{Horm67} {\sc L. H{\"o}rmander}: 
                 Hypoelliptic second order differential equations.
                 {\it Acta Math.}~119~(1967)~147--171.
                  \vs
 
\bibitem{HJ90} {\sc R. A. Horn, C. R. Johnson}:  {\it Matrix Analysis}, Cambridge University Press, Cambridge, 1990.
\vs 

\bibitem{Kog17} {\sc A.~\!E. Kogoj}:
                 On the Dirichlet Problem for hypoelliptic evolution equations: Perron-Wiener solution and a cone-type criterion. 
                 {\it J. Differential Equations}~262~(2017)~1524--1539.
                 \vs
  
\bibitem{KLT18} {\sc A.~\!E. Kogoj, E. Lanconelli, G. Tralli}:
                  Wiener-Landis criterion for Kolmogorov-type operators.
                 {\it Discrete Contin. Dyn. Syst. Ser. A}~38~(2018)~2467--2485.
                  \vs
       
\bibitem{L73} {\sc E. Lanconelli}: Sul problema di Dirichlet per l'equazione del calore. {\it Ann. Mat. Pura Appl.} (4) {\bf 97} (1973) 83--114.
\vs
  
\bibitem{LP94} {\sc E. Lanconelli, S. Polidoro}:
               On a class of hypoelliptic evolution operators.
               {\it Rend. Sem. Mat. Univ. Pol. Torino}~52~(1994)~29--63.
               \vs
   
\bibitem{LU10} {\sc E. Lanconelli, F. Uguzzoni}:
               Potential analysis for a class of diffusion equations: a Gaussian bounds approach.
               {\it J. Differential Equations}~248~(9)~(2010)~2329–2367.
               \vs
   
\bibitem{Lan69} {\sc E.~\!M. Landis}:
                 Necessary and sufficient conditions for the regularity of a boundary point for the Dirichlet problem for the heat equation.
                 {\it Dokl. Akad. Nauk SSSR}~185~(1969)~517–-520.
                 \vs
    
\bibitem{Luk74}   {\sc J.~Luke\v{s}}:
                  Th\'eor\`eme de {K}eldych dans la th\'eorie axiomatique de {B}auer des fonctions harmoniques.
                  {\it Czechoslovak Mathematical Journal }~24~(1)~(1974)~114--125.
                  \vs
    

\bibitem{Lun97} {\sc A.~Lunardi}: Schauder estimates for a class of       degenerate elliptic and parabolic operators with unbounded coefficients in~${\bf R}^n$.
     {\it Ann. Scuola Norm. Sup. Pisa Cl. Sci.}~24~(4)~(1997)~133--164.
       \vs
  
\bibitem{Man97}  {\sc M. Manfredini}:
                 The Dirichlet problem for a class of ultraparabolic equations.
                 {\it Adv. Differential Equations}~2~(1997)~831--866.
                 \vs
  
\bibitem{Mon96}  {\sc A. Montanari}:
                 Harnack Inequality for Totally Degenerate Kolmogorov-Fokker-Planck Operators.
                 {\it Bollettino U.M.I.}~10-B~(7)~(1996)~903--926.
                 \vs  
    
\bibitem{NS84}  {\sc P. Negrini, V. Scornazzani}: 
                Superharmonic functions and regularity of boundary points for a class of elliptic-parabolic partial differential operators. {\it Bollettino U.M.I. Analisi Funzionale e Applicazioni Serie VI}~3~(1984)~85--107.
                \vs
\end{thebibliography}
\end{document}